\documentclass[reqno,12pt]{amsart}
\usepackage[a4paper, margin=1in]{geometry}

\usepackage{amsmath,amssymb,amsthm,stmaryrd,bm}
\usepackage{verbatim}
\usepackage[english]{babel}
\usepackage[utf8]{inputenc}
\usepackage{hyperref}
\usepackage{enumitem}
\usepackage{cleveref}
\usepackage{multirow,float}
\usepackage{csquotes}
\usepackage{quiver}
\usepackage{rotating,subcaption}
\usepackage{sidecap}
\usepackage{refcount}
\usepackage{longtable}
\usepackage{adjustbox}

\usepackage{makecell}
\usepackage{cellspace} 
\setlist[enumerate,1]{label=(\roman*)}

\usepackage[doi=false,isbn=false,url=false,giveninits=true,style=alphabetic,maxbibnames=99]{biblatex}
\DeclareFieldFormat{shorthandwidth}{#1} %
\AtEveryBibitem{%
\ifentrytype{article}{
    \clearfield{number}
}{}
}
\renewbibmacro{in:}{%
  \ifentrytype{article}{}{\printtext{\bibstring{in}\intitlepunct}}}
\AtEveryBibitem{%
\ifentrytype{book}{
    \clearfield{pages}
}{}
\ifboolexpr{
  test {\iffieldundef{shorthand}}
}
{}
{
  \clearfield{labelalpha}
  \clearfield{extradate}
}
}

\makeatletter
\renewcommand\paragraph{\vskip2mm \textbf } 
\makeatother

\theoremstyle{plain}
\newtheorem{theorem}{Theorem}
\newtheorem{lemma}[theorem]{Lemma}

\newtheorem{proposition}[theorem]{Proposition}
\newtheorem{corollary}[theorem]{Corollary}

\newtheorem{question-conjecture}[theorem]{Question-Conjecture}

\newtheorem*{theorem*}{Theorem}
\newtheorem*{lemma*}{Lemma}
\newtheorem*{proposition*}{Proposition}
\newtheorem*{corollary*}{Corollary}
\newtheorem*{exercise*}{Exercise}
\newtheorem*{fact*}{Fact}
\newtheorem*{conjecture*}{Conjecture}

\theoremstyle{remark}
\newtheorem{remark}[theorem]{Remark}

\newtheorem{question}{Question}

\newtheorem*{remark*}{Remark}
\newtheorem*{question*}{Question}

\theoremstyle{definition}
\newtheorem{definition}[theorem]{Definition}

\newtheorem*{note*}{Note}
\newtheorem*{warning*}{Warning}
\newtheorem*{notation*}{Notation}

\newtheorem*{definition*}{Definition}
\newtheorem*{example*}{Example}

\newcommand{\A}{{\mathbb {A}}}
\newcommand{\C}{{\mathbb {C}}}

\newcommand{\F}{{\mathbb {F}}}
\newcommand{\G}{{\mathbb {G}}}
\renewcommand{\H}{{\mathbb {H}}}

\newcommand{\N}{{\mathbb {N}}}
\renewcommand{\P}{{\mathbb {P}}}
\newcommand{\Q}{{\mathbb {Q}}}

\newcommand{\R}{{\mathbb {R}}}
\newcommand{\BS}{{\mathbb {S}}}
\newcommand{\Z}{{\mathbb{Z}}}

\newcommand{\cD}{{\mathcal {D}}}

\newcommand{\cK}{{\mathcal {K}}}
\newcommand{\cL}{{\mathcal {L}}}

\newcommand{\cP}{{\mathcal {P}}}

\newcommand{\cS}{{\mathcal {S}}}

\DeclareMathOperator{\Aut}{Aut}
\DeclareMathOperator{\Bl}{Bl}
\DeclareMathOperator{\Cl}{Cl}
\DeclareMathOperator{\Gal}{Gal}
\DeclareMathOperator{\GL}{GL}
\DeclareMathOperator{\Hom}{Hom}

\DeclareMathOperator{\Nm}{Nm}

\DeclareMathOperator{\ord}{ord}
\DeclareMathOperator{\PGL}{PGL}
\DeclareMathOperator{\PSL}{PSL}

\DeclareMathOperator{\SL}{SL}

\newcommand\legendre[2]{\genfrac(){}{}{#1}{#2}}
\newcommand{\ad}{\mathrm{ad}}
\newcommand{\fin}{\mathrm{fin}}
\newcommand{\ov}{\overline}
\renewcommand{\phi}{\varphi}

\newcommand{\coloneqq}{:=}

\usetikzlibrary{patterns, patterns.meta}

\title[Galois Realisations of $\mathrm{PSL}_2(\mathbb{F}_{p^2})$ via non-unirat.\ Hilbert Irr.]{Galois Realisations of $\mathrm{PSL}_2(\mathbb{F}_{p^2})$ via non-unirational Hilbert Irreducibility}

\author{Julian Demeio}
\address{Julian Demeio \\ 
Institut f\"ur Algebra, Zahlentheorie und Diskrete Mathematik \\
Leibniz Universit\"at Hannover \\
Welfengarten 1  \\
Hannover \\
30167 \\
Germany
}

\author{Damián Gvirtz-Chen}
\address{Damián Gvirtz-Chen \\ 
Department of Mathematics \\
University of Glasgow \\ 
University Place \\
Glasgow \\
G12~8QQ \\
United Kingdom}

\begin{document}
\maketitle
\begin{abstract}
We establish non-unirational versions of Hilbert Irreducibility for all Hilbert modular surfaces which are of K3 type. As an application we prove new instances of the regular Inverse Galois Problem for the simple groups $\PSL_2(\F_{p^2})$ subject to congruence conditions on $p$.
\end{abstract}

\section{Introduction}

The following popular question, going back to at least Noether, has remained open for over a century:
\begin{question}[Inverse Galois Problem]
    Is every finite group $G$ the Galois group of a finite extension of $\Q$?
\end{question}

We also have the variant:
\begin{question}[Regular Inverse Galois Problem]
    Is every finite group $G$ the Galois group of a finite geometrically integral cover $C \to \P^1_{\Q}$?
\end{question}

By the Hilbert Irreducibility Theorem (a positive answer to) the regular Inverse Galois Problem implies (a positive answer to) the classical one. On the other hand, the regular version has various interesting implications: for instance, it gives not just one $G$-extension of $\Q$ but infinitely many linearly disjoint $G$-extensions over any number field by specialisation.

A fundamental case for the questions above is when $G$ is simple. Finite simple groups have been classified in the following families: groups of Lie type, alternating groups of degree at least $5$, cyclic groups and the $27$ sporadic groups (including the Tits group). The regular Inverse Galois Problem has long been solved for alternating groups, cyclic groups, and all but one of the sporadic groups \cite[Theorem II.10.3]{IGT}. So the infinite families that remain are those of Lie type. Arguably the simplest such family, and a testing ground for any new techniques, is the series $A_1(q)=\PSL_2(\F_{q})$ for a prime power $q=p^k$. In this case, the regular Inverse Galois Problem has been solved for $k=1$ when there exists $\ell\in\{2,3,5,7\}$ with $\legendre{\ell}{p}=-1$ \cite{Shih, Malle} (although non-regularly for all $p$ \cite{Zywina}) and for $k=2$ when $\legendre{5}{p}=-1$ \cite{Feit, Mestre, Mestre2}, when $\legendre{2}{p}=-1$ or  $\legendre{3}{p}=-1$ \cite{Shiina}, and finally when $\legendre{-3}{p}=-1$ or $\legendre{-7}{p}=-1$ \cite{DW} (although non-regularly for more $p$ \cite{DV}).

In this paper we introduce new methods for the Inverse Galois Problem to prove the following. Let
\[
\cL:=\{\ell\ \text { prime}: \ell \leq 41, \ell \neq 31\}.
\]
\begin{theorem}\label{Thm:IGP}
    For any prime $p$ for which there exists an $\ell \in \cL$ with the Kronecker symbol $\legendre{\ell}{p}$ equal to $-1$, the group $\PSL_2(\F_{p^2})$ is a \emph{regular} Galois group over $\Q$.\footnote{Concretely, this implies that $\PSL_2(\F_{p^2})$ is realisable as a regular Galois group over $\Q$ for all $p<3,267,289$ (compared to the previous implied bound $p<1,009$).}
\end{theorem}

In particular, this realises the group $\PSL_2(\F_{p^2})$ over any Hilbertian field $k$ for the primes $p$ above. For the primes $\ell \leq 17, \ell \neq 7,11$ we apply the classical Hilbert Irreducibility Theorem to the Hilbert modular surfaces for which Elkies and Kumar \cite{EK} have proved $\Q$-rationality. For the primes $\ell=7,11,19,23,29,37$ and $41$, an additional difficulty arises: the relevant Hilbert modular surfaces are not (uni)rational. We thus prove the theorem with the help of new versions of the Hilbert Irreducibility Theorem over non-unirational surfaces. More precisely, we show that the Hilbert Property holds for several low-discriminant Hilbert modular surfaces, i.e.\ their sets of rational points are not thin. These surfaces are K3 surfaces, and here our approach uses the multiple fibration method initially introduced by Corvaja-Zannier \cite{CZ}, and further studied by the first-named author \cite{Demeio} and by the second-named author in collaboration with Mezzedimi \cite{GCM}. This seems to be the first time that a non-unirational Hilbert Property is applied to the Inverse Galois Problem.

\begin{theorem}\label{Thm:K3}
 Let $X/\Q$ be the Hilbert modular surface of discriminant $D$ and genus $g$ at base level. If $X$ is birational to a K3 surface, then $X(\Q)$ is not thin.
\end{theorem}

In fact, we prove an even stronger statement, namely that $X(\Q(t))$ is non-thin in $X_{\bar \Q(t)}(\bar \Q(t))$ (Theorem \ref{Thm:K3regular}), thus producing an abundance of {\em rational curves}, and not just rational points, on the Hilbert modular surfaces under investigation.
The application to the Inverse Galois Problem is found through the following
\begin{theorem}\label{Thm:cover}
    Let $K$ be the real quadratic field of discriminant $D$ with ring of integers $\mathfrak o$, and let $X/\Q$ be the Hilbert modular surface of discriminant $D$ and genus $g$ at base level. For every rational prime $p$ that is inert in $K$, there exists a Galois geometrically integral cover $Y_p \to X$ with group $\PSL_2(\mathfrak o/p)$.
\end{theorem}

Compared to the well-known modular curve setting, the existence of the Galois cover which is geometrically integral may be surprising at first, see Remark~\ref{Rem:shih}.

Theorem \ref{Thm:IGP} is proven by combining Theorems \ref{Thm:K3} and \ref{Thm:cover}. 
The proofs of Theorems \ref{Thm:K3} and \ref{Thm:cover} rely on the theory of canonical models of compactified Shimura varieties as applied to Hilbert modular surfaces. More precisely, Theorem \ref{Thm:K3} relies on an analysis of the intersection graph of these modular surfaces, while Theorem \ref{Thm:cover} follows naturally by looking at the Shimura data involved. (Alternatively, Theorem \ref{Thm:cover} can be inferred by looking at the $p$-torsion representation of the abelian surfaces with real multiplication parametrized by Hilbert Modular Surfaces.) 

The particular Shimura varieties under our consideration {\em already have} many (smooth) rational curves that arise as special algebraic cycles coming from the moduli structure. Namely, these come from the boundary of the compactification and from embeddings of Shimura data. However, specialising the cover $Y_p \to X$ to {\em these} rational curves does not produce surjective representations $\Gal(\overline{\Q(t)}/\Q(t)) \to \PSL_2(\F_{p^2})$ precisely because of the special origin of these curves, making this ``naive'' specialisation \emph{not} sufficient to construct Galois representations of $\PSL_2(\F_{p^2})$. Nonetheless, these smooth rational curves give rise to many genus $1$ fibrations defined over the base field $\Q$ on the surface $X$. This geometric feature of $X$, which itself does not appear to have a modular interpretation, is what ultimately allows us to propagate these rational curves and obtain new ones for which the specialised representation instead {\em does} indeed surject. More formally, this ``propagation'' is expressed by the multiple fibration method used to prove Theorem \ref{Thm:K3}.

The Shimura theory is particularly explicit in our case but it would be interesting to see if a similar approach could apply to other settings.

It would also be interesting to study the Hilbert Property for honestly elliptic Hilbert modular surfaces, where the multiple fibration method does not apply. If instead of proving non-thinness of rational points the goal is to only specialise a specific Galois cover like in Theorem~\ref{Thm:cover}, then one may use less than the full Hilbert Property, although knowledge of the intersection graph of the modular surface would still be necessary. We will pursue this approach in a separate forthcoming work.

The article is organised as follows: In \S\ref{Sec:hilbert}, we review and extend the multiple fibration method to nonconstant $\Q(t)$-points. We develop the necessary Shimura theory of Hilbert modular surfaces in \S\ref{Sec:shimura} and construct the cover from Theorem~\ref{Thm:cover} in \S\ref{Sec:cover}. We then study the algebraic cycles on our surfaces in \S\ref{Sec:cycles}. Finally, we prove Theorems~\ref{Thm:IGP} and \ref{Thm:K3} in \S\ref{Sec:main}.

\paragraph{Acknowledgements}
The authors thank Gerard Van der Geer and John Voight for helpful comments. They gratefully acknowledge the hospitality of ICMS Edinburgh, the Center for Mathematical Sciences at Technion, and the University of Bath, where part of this research was carried out, and thank Daniel Loughran for the invitation to the latter. The second author was supported by UKRI award UKRI094.

\section{Notation and terminology}

An \emph{algebraic variety} is a reduced, separated scheme of finite type over a field $k$. A \emph{cover of algebraic varieties} is a finite dominant morphism $f:Y \to X$ where $Y$ is integral and both $X$ and $Y$ are normal. 

\vskip1mm

 A morphism $f:Y \to X$ between $k$-varieties is {\em unramified} at $y \in Y$ if the differential $\mathrm{d}f:T_{y}Y \to T_{f(y)}X$ is injective. Otherwise, it is {\em ramified}. The set of all ramified points is closed; we refer to it as the {\rm ramification locus} of $f$, and to its image under $f$ as the {\em branch locus}. The ramification locus of $f:Y \to X$ is cut out by the $0$-th Fitting ideal of $\Omega_{Y/X}$, and, as such, it is invariant under pullback of maps (i.e.\ base-change) \cite[Tag07Z6, Tag 0C3H]{stacksproject}, and so is its image. By Zariski's purity theorem \cite{Zariski}, the branch locus of a cover is a divisor when $X$ is smooth over $k$. 

\vskip1mm

For a ring $R$, we let $\PGL_2(R)$ and $\PSL_2(R)$ be the quotients of $\GL_2(R)$ and $\SL_2(R)$ by scalar matrices. 
For an $R$-module $M$, we define its {\em primitive} elements $M_{\text{prim}} \subset M$ to be its elements that are not divisible by non-invertible elements of $R$. For a projective $R$-module $M$, we denote by $\GL(M)$ the group of $R$-automorphisms of $M$. When $R$ is an integral domain with field of fractions $K$ and $M$ has rank $n$, fixing a $K$-basis of $M \otimes_R K$ yields an embedding $\GL(M) \hookrightarrow \GL_n(K)$. For any algebraic subgroup $G$ of $\GL_n$, we let $G(M)=\GL(M) \cap G(K)$.

\vskip1mm

If $G$ is any reductive group over $\Q$, we write $G^{\ad}=G/Z(G)$ for its \emph{adjoint} group, $G(\R)_+$ for the preimage of the identity component of $G^\ad(\R)$ in $G(\R)$ and $G(\Q)_+=G(\R)_+\cap G(\Q)$. 

\vskip1mm

Let $\cP$ denote the set of (finite) primes of $\mathbb Q$. Let $\A_\Q^\fin:=\prod'_{p \in \cP}\Q_p$ denote the ring of \emph{finite adeles} and $\A_\Z^\fin:=\prod_{p \in \cP}\Z_p$ the ring of \emph{finite integral adeles}.

\section{Genus $1$ fibrations and the Hilbert Property}\label{Sec:hilbert}
\begin{definition}
    Let $X$ be a smooth proper geometrically integral algebraic surface over a field $k$.
    \begin{enumerate}
    \item A \emph{genus $1$ fibration} is a morphism from $X$ to an irreducible curve $C$, whose generic fibre is a smooth, proper, geometrically integral curve of genus $1$.
    \item A curve $E\subset X$ is a \emph{genus $1$ configuration} if it appears in the Kodaira classification of singular fibres in genus $1$ fibrations \cite[Figure~4.4]{Silverman}. Equivalently, it is either an irreducible curve of arithmetic genus $1$ or all of its irreducible components are smooth rational curves and their intersection graph is an extended simply-laced Dynkin diagram (see Figure \ref{dynkin}).
    \end{enumerate}
\end{definition}

\begin{figure}[H]
\centering
    \begin{tikzpicture}[scale = 0.9]
    \filldraw 
    (-3,0) circle (2pt) node[above]{}  -- (-2.25,0) circle (2pt) node[above]{}  -- (-1.5,0) circle (2pt) node[above]{} ;
    
    \draw[dashed] (-1.5,0) -- (-1,0);
    \draw[dashed] (-0.5,0) -- (0,0);
    
    \filldraw (0,0) circle (2pt) node[above]{}     -- (0.75,0) circle (2pt) node[above]{}   -- (1.5,0) circle (2pt) node[above]{}  -- (-0.75,-0.75) circle (2pt) node[below]{} ;
    
    \draw (-0.75,-0.75) -- (-3,0);
    
    \draw (-0.75,0.75) node {$\tilde A_n$};

    \filldraw (3,1) circle (2pt)  node[left]{}  -- (3.75,0) circle (2pt)  node[left]{}  -- (4.5,0) circle (2pt)  node[below]{} ;
    \draw[dashed] (4.5,0) -- (5,0);
    \draw[dashed] (5.5,0) -- (6,0);
    
    \filldraw (6,0) circle (2pt)  node[below]{}  -- (6.75,0) circle (2pt)  node[right]{}  -- (7.5,1) circle (2pt)  node[left]{} ;
    \filldraw (7.5,-1) circle (2pt)  node[left]{}  -- (6.75,0);
    \filldraw (3,-1) circle (2pt)  node[left]{}  -- (3.75,0);
    
    \draw (5.25,0.75) node {$\tilde D_n$};

    \filldraw (9,0) circle (2pt) node[below]{}   -- (9.75,0) circle (2pt) node[below]{}   -- (10.5,0) circle (2pt) node[above]{}   -- (11.25,0) circle (2pt) node[below]{}   -- (12,0) circle (2pt) node[below]{}   ;
    \filldraw (10.5,0) -- (10.5,-0.75) circle (2pt) node[right]{}  ;
    \filldraw (10.5,-0.75) -- (10.5,-1.5) circle (2pt) node[right]{}  ;
    
    \draw (10.5,0.75) node {$\tilde E_6$};
  \end{tikzpicture}
  \end{figure}

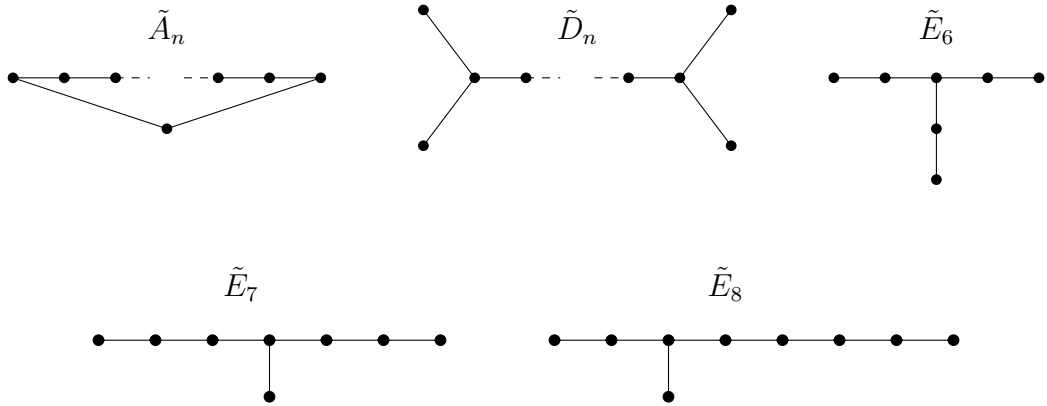
\begin{figure}[H]
\centering
\begin{tikzpicture}

    \filldraw (3,0) circle (2pt) node[below]{}   -- (3.75,0) circle (2pt) node[below]{}   -- (4.5,0) circle (2pt) node[below]{}   -- (5.25,0) circle (2pt) node[above]{}   -- (6,0) circle (2pt) node[below]{}   -- (6.75,0) circle (2pt) node[below]{}  -- (7.5,0) circle (2pt) node[below]{} ;
    \filldraw (5.25,0) -- (5.25,-0.75) circle (2pt) node[right]{}  ;

    \draw (4.875,0.75) node {$\tilde E_7$};

    \filldraw (9,0) circle (2pt) node[below]{}  -- (9.75,0) circle (2pt) node[below]{}  -- (10.5,0) circle (2pt) node[above]{} -- (11.25,0) circle (2pt) node[below]{} -- (12,0) circle (2pt) node[below]{}  -- (12.75,0) circle (2pt) node[below]{}  -- (13.5,0) circle (2pt) node[below]{}  -- (14.25,0) circle (2pt) node[below]{} ;
    \filldraw (10.5,0) -- (10.5,-0.75) circle (2pt) node[right]{} ;
    
    \draw (11.25,0.75) node {$\tilde E_8$};
\end{tikzpicture}
\caption{Extended simply-laced Dynkin diagrams. $\tilde A_n$ and $\tilde D_n$ have $n+1$ nodes.}
\label{dynkin}
\end{figure}

When $X$ is a K3 surface over $k$, then the base curve $C$ is always either a smooth projective conic without rational points or isomorphic to $\P^1_k$ \cite[Proposition~4.3.5]{Enriques}. Moreover any reduced fibre of a genus $1$ fibration is a genus $1$ configuration in the above sense, and vice versa, any genus $1$ configuration is a reduced fibre of a genus $1$ fibration \cite[Theorems 2.3.10 and 11.1.9]{Huybrechts}.  If the genus $1$ configuration is defined over $k$, then so is the induced fibration and $C\simeq\P^1_k$. 

\subsection{A criterion for the non-unirational Hilbert Property}

After Serre \cite[Definition~3.1.2]{Serre}, an algebraic variety $X/k$ is said to have the \emph{Hilbert Property} if $X(k)$ is not a thin set, i.e.\ if $X(k)$ is not a union of sets of the form:
\begin{enumerate}%
    \item $\pi(Y(k))$, for some cover $\pi:Y \to X$ of degree $\geq 2$;
    \item $W(k)$, for a proper Zariski-closed subvariety $W \subseteq X$.
\end{enumerate} 

Like the Hilbert Irreducibility Theorem, the Hilbert Property enables specialisation of the Galois group: If $\phi:Y\to X$ is a Galois cover of algebraic varieties with Galois group $G$, then for all $x\in X(k)$ outside a suitably chosen thin set, the fibre $\phi^{-1}(x)$ is irreducible and $\Gal(\phi^{-1}(x)/k)\simeq G$ \cite[Prop.~3.3.1]{Serre}. The Hilbert Property for $X$ guarantees that the set of such $x$ is non-empty (in fact Zariski-dense in $X$).

\vskip1mm

We have the following criterion, the multiple fibration method:

\begin{theorem}[{\cite[Thm.~1.1]{Demeio}}\footnote{Theorem 1.1 of \cite{Demeio} is stated only for number fields, but the same proof goes through for any finitely generated field $k$ of characteristic $0$, see e.g.\ \cite[Thm.~6.1]{GCM}.}]\label{Thm:Hilbertellptic}
    Let $X$ be a projective smooth geometrically connected surface over a finitely generated field $k$ of characteristic $0$, that is endowed with genus $1$ fibrations $\pi_1,\ldots,\pi_n:X \to \P^1, \ n \geq 2$, which are pairwise distinct up to automorphisms of the base. Let $Z \subseteq X_{\overline{k}}$ be the union of all irreducible divisors of $X_{\overline{k}}$ that, for each $i$, are contained in a fiber of $\pi_i$. If
    \begin{enumerate}
        \item $X(k)$ is Zariski-dense in $X$, and 
        \item $X_{\overline{k}} \setminus Z$ is simply connected,
    \end{enumerate}
    then $X$ has the Hilbert Property.
\end{theorem}

We refer to the irreducible divisors of $X_{\overline{k}}$ appearing in the theorem as \emph{common vertical divisors} of $\pi_1,\ldots,\pi_n.$ 
Since $n \geq 2$ and $\pi_i^{-1}(t_1) \cdot \pi_i^{-1}(t_2) > 0$ for all $i \neq j$, no common vertical divisor can be an irreducible fiber of any fibration. In particular, if $X$ is minimal, all such divisors are smooth geometrically rational $(-2)$-curves and their configuration has to be the disjoint union of finitely many simply-laced non-extended Dynkin diagrams of types $A_n$, $D_n$, $E_6$, $E_7$, $E_8$ (being irreducible components of reducible fibers of a minimal genus $1$ fibration). We call the divisor $Z$ the \emph{over-exceptional divisor} for $\pi_1,\ldots,\pi_n$.

\vskip1mm

The following lemma shall prove particularly useful in verifying the second hypothesis in Theorem \ref{Thm:Hilbertellptic}:
\begin{lemma}\label{Lem8}
Let $X$ be a smooth proper surface defined over an algebraically closed field $\bar k$ of characteristic $0$, and $\pi:X \to \P^1$ be a genus $1$ fibration. Let $Z \subset X$ be a divisor that is vertical for $\pi$ and such that:
\begin{enumerate}
    \item $\pi$ admits a simply connected fiber none of whose components lies in $Z$;
    \item each fiber of $\pi$ contains a component of multiplicity $1$ that does not lie in $Z$.
\end{enumerate}
Then $X \setminus Z$ is simply connected. 
\end{lemma}

(For a relatively minimal fibration, the simply connected fibers are precisely those not of type $I_n$.)

\begin{proof}
    Let $U:= X \setminus Z$, $V \to U$ be a connected finite étale cover, and $\psi:Y \to X$ be the relative normalisation of $X$ in $V$. Let $g \coloneqq \pi \circ \psi:Y \to \P^1$, and $Y \xrightarrow{h} C \xrightarrow{p} \P^1$ be its Stein factorisation. In particular, $p$ is finite, the fibers of $h$ are connected \cite[Théorème 4.3.1]{EGAIII.I}, and $C$ is integral and normal \cite[Lemme 8.8.6.1]{EGAII} (compare also with \cite[(6.3.9)]{EGAII}). From these properties, we also infer that the number of connected components $\#\pi_0(g^{-1}(t))$ of a fiber $g^{-1}(t)$ is $\#p^{-1}(t)$.
    
    Consider the commutative diagram
    \[
    \begin{tikzcd}
        Y \arrow[r, "h"]\arrow[d, "\psi"'] & C \arrow[d, "p"]\\ 
        X \arrow[r, "\pi"] &\P^1,
    \end{tikzcd}
    \]
    with diagonal composition $g$. The first hypothesis guarantees the existence of a simply connected fiber $\pi^{-1}(t)$ disjoint from $Z$, and so in particular over which $\psi$ is unramified. The cover $g^{-1}(t)=\psi^{-1}(\pi^{-1}(t)) \to \pi^{-1}(t)$ is thus unramified, and hence trivial, giving $\#p^{-1}(t)=\#\pi_0(g^{-1}(t)) = \deg \psi$. Since clearly $\deg p \leq \deg \psi$, this gives $\deg p=\deg \psi$.
        
    Assume now by contradiction $\deg p > 1$. Then some $t \in \P^1$ lies in the branch locus of $C \to \P^1$, and by pullback invariance of the branch locus, the fiber $\pi^{-1}(t)$ is contained in the branch locus of the projection $pr_1:X \times_{\P^1}C \to X$. Let $\xi \in \pi^{-1}(t)$ be the generic point of a component of multiplicity $1$ as in the second hypothesis. Then $f:X \to \P^1$ is smooth at $\xi$, and its base-change $X \times_{\P^1}C \to C$ is smooth near $pr_1^{-1}(\xi)$. As $C$ is regular, $X \times_{\P^1}C$ is regular there and coincides with its normalisation $\widetilde{X \times_{\P^1}C}$. Thus $\xi$ is contained in the branch locus of $\widetilde{X \times_{\P^1}C} \to X$. But $\xi\in U$, where $Y \to X$ is unramified (and hence any subcover, such as $\widetilde{X \times_{\P^1}C} \to X$), yielding a contradiction.
    \end{proof}

\subsection{Regular variants of the Hil\-bert Property}

Recall that a transcendental extension $k(V)/k$ is {\em regular} if $k$ is algebraically closed in $k(V)$. The following analog of the aforementioned \cite[Prop.~3.3.1]{Serre} serves as a useful tool to realise {\em regular} Galois representations.

\begin{proposition}\label{Prop331regular}
    Let $k$ be a field of characteristic $0$, and $k(V)$ be a regular finitely generated transcendental extension of $k$. Let $\phi:Y\to X$ be a  Galois cover with group $G$ of geometrically integral $k(V)$-varieties. Let $S \subset X(k(V))$ be a subset whose natural image in $X(\bar k(V))=X_{\bar k(V)}(\bar k(V))$ is a non-thin set of points of the base-changed variety $X_{\bar k(V)}$. Then there are Zariski-dense $x \in S$ for which the fibre $\phi^{-1}(x)$ is irreducible with regular function field over $k$, and with $\Gal(\phi^{-1}(x)/k(V))\simeq G$.
\end{proposition}

\begin{proof}
    The same argument of \cite[Prop.~3.3.1]{Serre} works here. Namely, take $x \in S$ whose natural image $\bar x$ in $X_{\bar k(V)}(\bar k(V))$ is not contained in the union $\bigcup_{H \lneq G}(Y_{\bar k(V)}/H)(\bar k(V))$. Following the argument of {\em loc.cit.}, we see that the fiber $\bar\phi^{-1}(\bar x)$ of $\bar x$ under the base-changed morphism $\bar \phi:Y_{\bar k(V)} \to X_{\bar k(V)}$ is irreducible and $\Gal(\bar \phi^{-1}(\bar x)/\bar k(V))\simeq G$.     
    Denoting by $L/k(V)$ the étale-algebra such that $\operatorname{Spec} L=\phi^{-1}(x)$, we have $\operatorname{Spec} (L \otimes_{k(V)}\bar k(V))=\bar \phi^{-1}(\bar x)$. Thus $L \otimes_{k(V)}\bar k(V)$ is a field, and its subalgebra $L$ is as well. The extension $L/k(V)$ is Galois with group $G$ by construction.
\end{proof}

The following theorem is a modification of Theorem~\ref{Thm:Hilbertellptic}.

\begin{theorem}\label{Thm:Hilbertellpticregular}
    Let $X$ be a K3 surface over a field $k$ of characteristic $0$, that is endowed with genus $1$ fibrations $\pi_1,\ldots,\pi_n:X \to \P^1_k, \ n \geq 2$, which are pairwise distinct up to automorphisms of the base. Let $Z \subseteq X_{\overline{k}}$ be the union of all irreducible divisors of $X_{\overline{k}}$ that, for each $i$, are contained in a fiber of $\pi_i$. If
    \begin{enumerate}
        \item the image of the non-constant elements of $X(k(t))$ (where $p \in X(k(t))$ is constant if the corresponding map $\P^1_k \to X$ is) under the identification $X(k(t))=X_{k(t)}(k(t))$ is Zariski-dense in the base-change $X_{k(t)}$, and 
        \item $X_{\overline{k}} \setminus Z$ is simply connected,
    \end{enumerate}
    then the natural image of $X(k(t))$ in $X_{\bar k(t)}$ is non-thin.
\end{theorem}

One often wishes to take $Z$ minimal, so we record the following immediate corollary:
\begin{corollary}\label{Cor:Hilbertellpticregular}
    Let $X$ be a K3 surface over a field $k$ of characteristic $0$, that is endowed with at least two genus $1$ fibrations $X \to \P^1_k$ distinct up to isomorphism of the base. Let $Z \subseteq X_{\overline{k}}$ be the union of all irreducible divisors of $X_{\overline{k}}$ that are vertical with respect to {\em all} genus $1$ fibrations of $X$. If
    \begin{enumerate}
        \item the image of the non-constant elements of $X(k(t))$ (where $p \in X(k(t))$ is constant if the corresponding map $\P^1_k \to X$ is) under the identification $X(k(t))=X_{k(t)}(k(t))$ is Zariski-dense in the base-change $X_{k(t)}$, and 
        \item $X_{\overline{k}} \setminus Z$ is simply connected,
    \end{enumerate}
    then the natural image of $X(k(t))$ in $X_{\bar k(t)}$ is non-thin.
\end{corollary}
\begin{proof}
    For every finite collection $\Pi$ of genus $1$ fibrations $X \to \P^1_k$, define $Z_{\Pi}$ as the union of irreducible divisors that are contained, for each $\pi \in \Pi$, in a fiber of $\pi$. We have
    \[
    Z=\bigcap_{\Pi}Z_{\Pi}.
    \]
    Letting $\pi_1,\pi_2$ be the two genus $1$ fibrations given at the beginning, certainly $Z\subset Z_{\pi_1,\pi_2}$, and $Z_{\pi_1,\pi_2}$ is a finite union of divisors as $\pi_1$ and $\pi_2$ generate distinct pencils. Thus the (possibly infinite) intersection above happens in the finite union of divisors $Z_{\pi_1,\pi_2}$, and so there exists $\Pi$ with $Z_{\Pi}=Z$. Let now $\{\pi_1,\pi_2,\pi_3,\ldots,\pi_n\}:=\Pi$ and apply Theorem \ref{Thm:Hilbertellpticregular}.
\end{proof}

We call the $Z$ appearing in Corollary \ref{Cor:Hilbertellpticregular} the {\em over-exceptional divisor} of $X$.

\vskip1mm

To prove Theorem \ref{Thm:Hilbertellpticregular}, we adapt the proof in \cite{Demeio}. The only issue in this adaptation is that the proof in \cite{Demeio} relies on an application of Faltings’ theorem to some curves which might {\em a priori} be constant in our new setting, and Manin--Grauert's theorem (Theorem \ref{Thm:DFMG} below), which is ``Faltings' analogue for function fields'', does not apply to constant curves. However, we manage to prove that the constancy never happens when $X$ is a K3 surface.

\vskip1mm

A curve $C/k(V)$ is {\em constant} over $k$ if there exists a curve $C_0/k$ such that $C\cong C_0 \otimes_kk(V)$.

\begin{theorem}\label{Thm:IsotrRational}
    Let $\pi:X \to \P^1_{\bar k}$ be a genus $1$ fibration, with $X$ a K3 surface defined over an algebraically closed field $\bar k$ of characteristic $0$, and let $u$ be the affine coordinate of $\P^1_{\bar k}$. Then, for all rational extensions $\bar k(t)/\bar k(u)$, the base-change $X_{\bar k(t)}$ of the generic fiber $X_{\bar k(u)}\coloneqq \pi^{-1}(\operatorname{Spec} \bar k(u))$ is a non-constant curve.
\end{theorem}

\begin{proof}
    It suffices to prove that the Jacobian curve of $X_{\bar k(t)}$ is non-constant, so we may replace $\pi$ with its Jacobian fibration, which is again a K3 surface \cite[Sec.~11.4]{Huybrechts}, and assume that $\pi$ has a section. We may also assume that $\pi:X \to \P^1_{\bar k}$ is isotrivial, as the statement is automatic otherwise. The elliptic surface $X$ is then (a regular relatively minimal model of) a twist of a constant surface $E \times \P^1_{\bar k}$ by some cocycle $\alpha \in H^1(\bar k(u),\Aut(E))=\Hom(\Gamma_{\bar k(u)},\Aut(E))$, where $\Aut(E)\in \{\Z/n\Z: n = 2,3,4,6\}$ denotes the automorphism group of the elliptic curve $E$. Let $F \subset \overline{k(u)}$ be the subfield fixed by $\operatorname{Ker} \alpha$, so that $G \coloneqq \Gal(F/\bar k(u))= \operatorname{Im} \alpha \subset \Aut(E)$, and $X$ is birational to the diagonal quotient $(E \times C)/G$, where $C/\bar k$ is a curve with function field $F$. Since $X$ is simply connected, we have $\#G >1.$

    If $X_{\bar k(t)}$ is constant for some rational extension $\bar k(t)/\bar k(u)$, then the restriction of $\alpha$ to $\Gamma_{\bar k(t)}$ is trivial, i.e.\ $F \subset \bar k(t)$ and $F=\bar k(v)$ for some $v \in F$ by L\"uroth's theorem, i.e.\ $C$ is a rational curve. Now, the smooth fibers of the projection $(E \times C)/G \to E/G$ are forms of $C$, and are in particular geometrically rational. The curve $E/G$ has genus $0$ since it is dominated by the elliptic curve $E$ via a ramified map. It follows that $(E \times C)/G$ is birational to a conic bundle over $\P^1_{\bar k}$, and is thus rational by Tsen's theorem, giving a contradiction as $X$ is a K3 surface.
\end{proof}

In place of Faltings' theorem, we shall use Manin--Grauert's theorem \cite{Samuel}.

\begin{theorem}[Manin--Grauert]\label{Thm:DFMG}
    Let $C/\bar k(V)$ be a geometrically irreducible curve of genus $>1$ defined over a finitely generated field $\bar k(V)$ over an algebraically closed field $\bar k$ of characteristic $0$. If $C$ is non-constant over $\bar k$, then $C(\bar k(V))$ is finite.
\end{theorem}
We shall also need:
\begin{lemma}\label{Lem:Dominant}
    Let $C_1 \to C_2$ be a dominant map of geometrically irreducible $\bar k(V)$-curves. If $C_2$ is non-constant over $\bar k$ of genus $\geq 1$, then so is $C_1$.
\end{lemma}
\begin{proof}
    We may view $\bar k(V)$ as the function field of a smooth projective variety, and use this to define (ample) heights on $C_1(\overline{k(V)})$ and $C_2(\overline{k(V)})$, as in \cite[11-12]{SerreMW}.

    Suppose that $C_1=C_{1,0} \otimes_{\bar k} \bar k(V)$ is constant. The infinitely many points $P \in C_{1,0}(\bar k)$ induce infinitely many $\bar k(V)$-points $P \otimes \bar k(V)$ in $C_1(\bar k(V))$. These have bounded height, and so project to infinitely many points in $C_2(\bar k(V))$ of bounded height. Since $C_2$ is non-constant, Manin--Grauert's theorem gives a contradiction if $g(C_2)>1$, and \cite[Thm.~6.5.3]{Lang} (or \cite[Thm.~III.5.4]{Silverman}) gives a contradiction if $g(C_2)=1$. \qedhere
\end{proof}

\begin{proof}[Proof of Theorem \ref{Thm:Hilbertellpticregular}]
    The proof is simply an adaptation of \cite[Thm.~1.1]{Demeio}, using Manin--Grauert's theorem instead of Faltings' theorem, and Theorem \ref{Thm:IsotrRational} and Lemma \ref{Lem:Dominant} serve to ensure that we never fall in the constant case. We explain how one should modify the proof in \cite{Demeio}, and refer to {\em loc.cit.} for the rest of the details.

    \vskip1mm

    Recall that a cover $\phi: Y \to X$ is {\em vertically ramified} with respect to a fibration $\pi:X \to \P^1$ if the branch divisor of $\phi$ is contained in a union of finitely many fibers of $\pi$, and is {\em horizontally ramified} otherwise.

    \vskip1mm

    Let us prove the following fact, which replaces the application of Faltings' theroem in \cite{Demeio}:
    \begin{center}
        $(\star)$ For a cover $\phi: Y \to X$ that is horizontally ramified with respect to a genus $1$ fibration $\pi:X \to \P^1$, there exists a non-empty open $U \subset X$ such that for all non-constant $p \in U(k(t))$, the curve $(\pi \circ \phi)^{-1}(\pi(p))$ has finitely many $\bar k(t)$-points.
    \end{center}

    In fact, take as $U$ the complement of all singular fibers of $\pi$. Then for all non-constant $p \in U(k(t))$, the image of the corresponding morphism $\rho:\P^1_k \to U$ is a rational curve not contained in the singular fibers of $\pi$, and thus the composition $f \coloneqq \pi \circ \rho:\P^1_k \to X \to \P^1_k$ is surjective and defines a rational extension $k(t)/k(u)$, where $u=f(t)$ is the affine coordinate on the base $\P^1_k$ of the fibration $\pi$. The fiber $\pi^{-1}(\pi(p))$ is the base-change of the generic fiber of $\pi$ along the extension $k(t)/k(u)$, and thus is non-constant over $\bar k$ by Theorem~\ref{Thm:IsotrRational}. 
    
    Now, the morphism $(\pi \circ \phi)^{-1}(\pi(p)) \to \pi^{-1}(\pi(p))$ is ramified. The $k(t)$-curve $(\pi \circ \phi)^{-1}(\pi(p))$ is not necessarily irreducible, but  all of its irreducible components are ramified over $\pi^{-1}(\pi(p))$ (it is clear that at least one component must ramify, see the argument in \cite[10270]{Demeio}\footnote{Or p.\ 12 in the ArXiv version.} for why all of them do), and thus have genus $>1$. Moreover, each of these components is non-constant over $\bar k$ by Lemma \ref{Lem:Dominant}, and $(\star)$ now follows from Manin--Grauert.

    To conclude we now operate as in {\em loc.cit}. Namely, take finitely many covers $\phi_j:Y_j \to X, j=1,\ldots,m$ of degree $\geq 2$ such that 
    \[
    X(k(t)) \setminus \cup_j \phi_j(Y_j(\bar k(t)))
    \]
    is not Zariski--dense in $X_{\bar k(t)}$. Choose a non-empty Zariski--open $U \subset X_{k(t)}$ such that
    \[
    U(k(t)) \subset \cup_j \phi_j(Y_j(\bar k(t))),
    \]
    and shrink $U$ so that $(\star)$ holds with respect to all pairs of $\phi_j$ and $\pi_i$, and, using \cite[Lemma 3.2]{Demeio} (based on Merel's theorem, which holds for the finitely generated field $k(t)$)\footnote{Lemma 3.2 in \cite{Demeio} is formulated for number fields, but the proof only uses Merel's theorem, and thus it also holds for any finitely generated field.}, further shrink $U$ so that for all $p \in U(k(t))$, the fibers $\pi_i^{-1}(\pi_i(p))$ are smooth and have infinitely many $k(t)$-points for all $i$.

    Consider now any non-constant $p \in U(k(t))$ (this exists by the first hypothesis), and let $C_p$ be the $k(t)$-curve $\pi_n^{-1}(\pi_n(p))$. We have
    \begin{multline*}
    p \in C_p(k(t)) \cap U \subset
    \bigcup_{\substack{\phi_j \text{ vertically}\\ \text{ ramified} \text{ wrt }\pi_n}} \varphi_j\left(\phi_j^{-1}(C_p)(\bar k(t))\right)
    \ \cup \bigcup_{\substack{\phi_j \text{ vertically}\\ \text{ ramified} \text{ wrt }\pi_n}} \varphi_j\left(\phi_j^{-1}(C_p)(\bar k(t))\right).
    \end{multline*}

    The second union is a finite set by $(\star)$, and we infer that all but finitely many elements of $C_p(k(t))$ are contained in the first union. The curve $C_p$ is a non-constant $k(t)$-curve of genus $1$, and thus by the weak Mordell--Weil theorem, for each $\phi_j$ that is vertically ramified with respect to $\pi_n$, $\varphi_j\left(\phi_j^{-1}(C_p)(\bar k(t))\right) \cap C_p(k(t))\subset C_p(k(t))$ is either empty or a finite index coset. It then follows from \cite[Lem.~4.1]{Demeio} that $C_p(k(t))$ is fully contained in the first union. 

    Proceeding with an easy induction on $n$, we reduce to the case where all $\phi_j$ are vertically ramified with respect to all $\pi_i$, but there are no such covers by assumption, giving the desired contradiction. \qedhere
\end{proof}

\section{Preliminaries on Hilbert modular surfaces}\label{Sec:shimura}
We follow \cite{VDG}. Let $K=\Q(\sqrt{D})$ be a real quadratic field of discriminant $D$, and let $\mathfrak o$ be its ring of integers. We denote conjugation in $K$ by $a \mapsto a'$. 

We write $\H^+$ and $\H^-$ for the upper and lower half-plane. For each genus $g \in \Cl^+K/2$ of $\mathfrak o$ we may pick a representative fractional ideal $\mathfrak a$ (of norm $A$) and define, for each finite index subgroup $\Gamma \subseteq \PSL(\mathfrak o \oplus \mathfrak a)$, the \emph{Hilbert modular surface}
\[
X_g(\Gamma):=\Gamma \backslash (\H^+)^2,
\]
where the quotient is by the properly discontinuous action:
\[
\begin{pmatrix}
    a & b \\ c&d
\end{pmatrix}\cdot (z_1,z_2) = \left( \frac{az_1+b}{cz_1+d} , \frac{a'z_2+b'}{c'z_2+d'} \right).
\]
Let $\Lambda=\mathfrak o \oplus \mathfrak a$. We call the surface $X_g:= X_g(\PSL(\Lambda))$ of \emph{base level}.\footnote{One could take further quotients of $(\H^+)^2$, the maximum being by the symmetrised Hurwitz-Maass extension of $\PSL(\Lambda)$, but we shall not be interested in these in this paper.} The isomorphism class of $X_g$ is independent of the choice of representative $\mathfrak a$. 

For each $N \in \N$, we have an isomorphism of $\mathfrak o$-modules $\Lambda/N\cong (\mathfrak o/N)^2$, and thus a map 
$$
\PSL(\mathfrak o \oplus \mathfrak a)=\PSL(\Lambda) \twoheadrightarrow \PSL(\Lambda/N) \cong \PSL_2(\mathfrak o/N),
$$
which is surjective by the strong approximation theorem for $\SL_{2,K}$ \cite[Sec.~7.4]{PR}. We denote its kernel by $\PSL(\Lambda,N)$. We let $X_g(N):=X_g(\PSL(\Lambda,N))$. The cover
\begin{equation}\label{Eq:coverC}
    X_g(N) \to X_g
\end{equation}
is  $\PSL_2(\mathfrak o/N)$-Galois.

So far we worked analytically, but we wish to define this cover over $\Q$. We do so via Shimura theory.

\subsection{Hilbert modular surfaces and rational models}\label{SSec:HMSQ}

We provide here a ready-to-use summary for the Shimura theory we need. We follow \cite{Rapoport}, \cite{Deligne}.

\paragraph{Arithmetic models.} Let $G$ be the fibered product:
\[
\begin{tikzcd}
G \arrow[d] \arrow[r] & R_{K/\mathbb{Q}} \GL(2) \arrow[d, "\det"] \\
\G_m \arrow[r, hook]  & R_{K/\mathbb{Q}}\G_m                   
\end{tikzcd}.
\]

Let $\BS\to(R_{K/\mathbb{Q}}\GL(2))_\R\cong\GL(2)_\R^2$ be the morphism of algebraic groups given by
$a+bi\mapsto\left(\begin{pmatrix}
    a & -b \\ b & a
\end{pmatrix}, \begin{pmatrix}
    a & -b \\ b & a
\end{pmatrix}\right),$ where $\BS=R_{\C/\R}\G_m$ is the Deligne torus. This map factors through a morphism $h:\BS\to G_\R$. Writing $\cK_\infty$ for the centraliser of $h$ in $G(\R)$, the $G(\R)$-conjugacy class of $h$ is $D\cong G(\R)/\cK_\infty\cong (\H^+)^2\cup(\H^-)^2$.

The pair $(G,D)$ is a Shimura datum in the sense of \cite[Conditions 2.1.1(1-3)]{Deligne2}. This datum defines the inverse system of complex orbifolds running over compact open subgroups $\cK_f\subset G(\A_\Q^\fin)$:
\[
X_{\cK_f}=G(\Q)\backslash D\times G(\A_\Q^\fin)/\cK_f=G(\Q)\backslash G(\A_\Q)/\cK_f\cdot\cK_\infty,
\]
and:
\begin{align*}
X'_{\cK_f}&=\text{ minimal compactification of } X_{\cK_f} \text{ after Baily--Borel}, \\
\tilde X_{\cK_f}&= \text{ minimal resolution of cusp and elliptic singularities of } X'_{\cK_f},\\
Y_{\cK_f}&= \text{ minimal smooth projective model of } \tilde X_{\cK_f}.
\end{align*}

The complex orbifolds $X_{\cK_f}$ are \emph{not} (necessarily) \emph{connected}. They are defined algebraically over the reflex field of $(G,D)$, which can be shown to be $\Q$, in the following sense (see \cite[p.237, ii, b]{VDG} for an analogous situation\footnote{In \cite{VDG} the Shimura theory is formulated using $R_{K/\mathbb{Q}} \GL(2)$ which corresponds to the quotient $\PGL(\Lambda)\backslash (\H^+\cup\H^-)^2$ rather than $\Gamma\backslash(\H^+)^2$. The former is at most a double quotient of the latter.}, and \cite{Rapoport}, \cite{Deligne}, \cite{Deligne2} and Keel--Mori's theorem \cite{Conrad} for a proof in our setting):
\begin{enumerate}
    \item for every compact open $\cK_f \subseteq G(\A_{\Q}^{\fin})$, there is a $\Q$-scheme whose analytification is $X_{\cK_f}$ (the {\em canonical model} of $G(\Q)\backslash G(\A_\Q)/\cK_f\cdot\cK_\infty$) -- we denote this scheme also by $X_{\cK_f}$;
    \item if $\cK_f$ is a normal subgroup of $\cL_f$ then $\cL_{f} / \cK_f$ acts on
    $X_{\cK_f}$ and \[X_{\cK_f} /\left(\cL_f / {\cK}_f\right) \cong X_{\cL_f}.\]
\end{enumerate}

Analytically, the $\cL_{f} / \cK_f$-action of point (ii) is defined via left action on the double quotient $G(\Q)\backslash G(\A_\Q)/\cK_f\cdot\cK_\infty$.

\paragraph{Connected components.}
The connected components of $X_{\cK_f}$ are (this is analogous to \cite[Proposition I.7.2]{VDG}):
\begin{equation}\label{Eq:adelic}
    G(\Q)\backslash G(\A_\Q)/\cK_f\cdot \cK_{\infty}= \bigsqcup_{j=1}^r (G(\Q)_+ \cap \gamma_j\cK_f\gamma_j^{-1}) \backslash (\H^+)^2,
\end{equation}
for any choice of points $\gamma_1,\ldots,\gamma_r \in G(\A_\Q^{\fin})$ such that $(\gamma_1,1),\ldots,(\gamma_r,1) \in G(\A_\Q^{\fin}) \times G(\R)$ are representatives of the connected components $\pi_0(G(\Q)\backslash G(\A_\Q)/\cK_f\cdot\cK_\infty)$. Let $h_T:\BS\to (\G_m)_\R,z\mapsto |z|^2$. The pair $(\G_m,\{\pm 1\}=\{\pm h_T\})$ is a zero-dimensional Shimura datum (in the extended sense of \cite[p. 62-63]{Milne}) and the determinant map induces a morphism $(G,X)\to(\G_m,\{\pm 1\})$, which on the connected components is a bijection of $\Gamma_\Q$-sets
\[
\Q^\times  \backslash\A^{\times}_\Q  / \det(\cK_f)\R_{>0}^\times\simeq \pi_0(G(\Q)\backslash G(\A_\Q)/\cK_f\cdot \cK_{\infty}),
\]
so we could equivalently say that $\det(\gamma) \in \A^{\times}_K$ varies in a set of representatives for $\Q^\times  \backslash\A^{\times}_\Q  / \det(\cK_f)\R_{>0}^\times$.

\paragraph{Galois action on $\pi_0$.} The Galois action of $\Gamma_\Q$ on the components of $X_{\cK_f}$ is given by the Deligne--Serre theorem (see \cite[p.\ 238]{VDG}, or \cite[\S13]{Milne}):
\begin{equation}\label{DS}
    \operatorname{Gal}\left(\mathbb{Q}^{\mathrm{ab}} / \mathbb{Q}\right) \simeq \pi_0\left(\Q^{\times}  \backslash \mathbb{I}_\Q\right) =\pi_0\left(G({\mathbb{Q}})\backslash G(\mathbb{A}_{\mathbb{Q}})\right)
\end{equation}
where the isomorphism is the reciprocity isomorphism of class field theory, and the last identification is induced by the determinant
\[
\det: G({\mathbb{Q}})\backslash G(\mathbb{A}_{\mathbb{Q}}) \simeq \Q^{\times}  \backslash \mathbb{I}_\Q.
\]

\paragraph{Examples.~}Let $\cK_{g} \subseteq G(\A_\Q^\fin)$ be the \emph{maximal} open compact subgroup associated to $\Lambda$, i.e.:
\[
\cK_{g}:=\prod_{v  < \infty} G(\Lambda \otimes_{\mathfrak o} \mathfrak o_v).
\]
Since $G(\Lambda \otimes_{\mathfrak o} \mathfrak o_v) \simeq G(\mathfrak o_v)$ for all $v < \infty$ and $\det G(\mathfrak o_v)=\mathfrak o_v^{\times}$, $\det(\cK_{\mathfrak a})=(\A^{\times}_{\Z})^{\fin}$ and $X_{\cK_{g}}$ is connected \cite{Rapoport}. 
Thus \eqref{Eq:adelic} reduces to:
\[
G(\Q)\backslash G(\A_\Q)/\cK_{\mathfrak a}\cdot \cK_{\infty} \cong (G(\Q)_+ \cap\cK_{\mathfrak a}) \backslash (\H^+)^2 = \PSL(\Lambda) \backslash (\H^+)^2,
\]
giving an identification $X_{\cK_{g}} = X_g.$ We also write $X'_g$, $\tilde X_g$ and $Y_g$ for $X'_{\cK_{\mathfrak a}}$, $\tilde X_{\cK_{\mathfrak a}}$ and $Y_{\cK_{\mathfrak a}}$. 

For each $N \in \Z_{>0}$, we have a surjective modulo $N$ reduction map $\cK_{\mathfrak a} \to G(\Lambda/N)\cong G(\mathfrak o/N)$. We define 
\[
\cK_{g}(N)= \operatorname{Ker} \left( \cK_{\mathfrak a} \to G(\Lambda/N) \right).
\]
Since $\cK(N)$ is normal in $\cK_{\mathfrak a}$, and the representatives $\gamma$ defined above may always be chosen in $\cK_{g}$, \eqref{Eq:adelic} gives an identification of $G(\Q)\backslash G(\A_\Q)/\cK_N\cdot \cK_{\infty}$ with $\phi(N)$ copies of $(G(\Q)_+ \cap \cK_N) \backslash (\H^+)^2=\SL(\Lambda,N) \backslash (\H^+)^2=\PSL(\Lambda,N) \backslash (\H^+)^2.$ 

Summarizing:
\[
X_N:=X_{\cK_N} \to X_{g}
\]
gives an arithmetic realisation of the cover \eqref{Eq:coverC}. However, $X_N$ is not geometrically connected. In fact, $\det(\cK_{\mathfrak a}(N))=\prod_{p\nmid N}\Z_p^\times \times \prod_{p|N}(1+p^{\ord_p(N)}\Z_p)$, hence $\Gamma_\Q$ acts on the set of geometric connected components of $X_N$ via the cyclotomic character $\chi:\Gamma_{\Q} \to (\Z/N\Z)^\times$.

\section{Construction of the $\operatorname{PSL}_2(\F_{p^2})$-cover}\label{Sec:cover}

We fix in this section a genus $g$ of the field $K$, and a module $\Lambda=\mathfrak o \oplus \mathfrak a$ representing it.

Let $N \in \N$ and $\cK_N\subseteq G(\A_{\Q}^{\fin})$ be the open-compact subgroup defining the full level-$N$ cover:
\[
f_N: G(\Q)\backslash G(\A_{\Q}) / \cK_N\cdot\cK_{\infty} \to G(\Q)\backslash G(\A_{\Q}) / \cK_{\mathfrak a}.
\]
This cover is endowed with a natural right $G(\mathfrak o_K/N)(= \cK_{\mathfrak a}/\cK_N)$-action. 

As described in Subsection \ref{SSec:HMSQ}, there exists a $\Q$-surface $X_N$ whose analytification is $G(\Q)\backslash G(\A_\Q) / \cK_N\cdot\cK_\infty$, such that $f_N:X_N \to X_g$ is also defined over $\Q$ and such that the $G(\mathfrak o_K/N)$-action commutes with the $\Gamma_\Q$-action. Moreover, $X_N/G(\mathfrak o_K/N)=X_g$.

The kernel of the $G(\mathfrak o_K/N)$-action is the image of $\cK_{\mathfrak a} \cap Z_{G}(K) = \{\pm 1\}$ in $G(\mathfrak o_K/N)$. Thus $X_N \to X_g$ is a $G_N$-Galois cover with $G_N=G(\mathfrak o_K/N)/\{\pm 1\}$, inducing a representation:
\[
\rho:\Gamma_{\Q(X_g)} \to G_N.
\]

\begin{lemma}\label{Lem:diagram}
    We have a commutative diagram with exact rows:
    \[
    \begin{tikzcd}
        1 \arrow[r] & \Gamma_{\C(X_g)}  \arrow[r] \arrow[d, "\rho_{\C}"] & \Gamma_{\Q(X_g)}  \arrow[r] \arrow[d, "\rho"]       & \Gamma_{\Q} \arrow[r] \arrow[d, "\chi"]         & 1 \\
        1 \arrow[r] & \SL_2(\mathfrak o_K/N)/\{\pm 1\} \arrow[r]                  & G(\mathfrak o_K/N)/\{\pm 1\} \arrow[r, "\det"] & (\Z/N\Z)^\times \arrow[r] & 1
    \end{tikzcd},
    \]
    where $\chi:\Gamma_{\Q} \to (\Z/N\Z)^\times$ is the cyclotomic character and $\rho_{\C}$ is surjective.
\end{lemma}
\begin{proof}
    The exactness of the first row follows from Galois theory.

    Recall that the cover $X_{N}(\C) \to X_g(\C)$ splits into $\phi(N)$ copies of the analytic cover
    \[
    \PSL(\Lambda,N) \backslash (\H^+)^2 \to \PSL(\Lambda) \backslash (\H^+)^2, 
    \]
    whose Galois group is $\PSL(\Lambda)/\PSL(\Lambda,N) \cong \SL_2(\mathfrak o_K/N)/\{\pm 1\}.$ Since the $\SL_2(\mathfrak o_K/N)$-action here is compatible with the $G(\mathfrak o_K/N)$-action on $X_{N}$, this shows that $\rho(\Gamma_{\C(X_g)})=\SL_2(\mathfrak o_K/N)/\{\pm 1\}.$ We define $\rho_{\C} = \rho|_{\Gamma_{\C(X_\Gamma)}}$.

    The just proven identity implies that there is an induced quotient action by $(\Z/N\Z)^\times\cong G(\mathfrak o_K/N)/\SL_2(\mathfrak o_K/N)$ on $\pi_0(X_{N})$. By Deligne--Shimura, this action coincides with the Galois action of $\Gamma_{\Q}$, proving the commutativity of the last square.
\end{proof}

Take now $N=p$ an inert prime in $K$. Projectivising $\rho$ by quotienting by scalar matrices, we get:
\begin{proof}[Proof of Theorem~\ref{Thm:cover}]
    Consider the composition:
    \[
    \rho':\Gamma_{\Q(X_\Gamma)} \xrightarrow{\rho} G(\mathfrak o_K/N)/\{\pm 1\} \to \PGL_2(\mathfrak o_K/N)=\PGL_2(\F_{p^2}).
    \]
    By Lemma \ref{Lem:diagram}, the determinant $\det \rho'$ is the image of the cyclotomic $p$-character under $\F_p^\times \to \F_{p^2}^\times/\F_{p^2}^{\times2}$. Since all elements of $\F_p^\times$ are squares in $\F_{p^2}$, this determinant is trivial, and thus $\rho'$ takes values in $\PSL_2(\F_{p^2})=\SL_2(\F_{p^2})/\{\pm 1\}$.
    Let $Y_N$ be the cover defined by the kernel of $\rho'$. Since the restriction $\rho_{\C}$ of $\rho'$ on $\Gamma_{\C(X_g)}$ surjects onto $\PSL_2(\F_{p^2})$, $(Y_N)_{\C}$ is geometrically integral and we are done.
\end{proof}

\begin{remark}\label{Rem:shih}
Theorem~\ref{Thm:cover} should be compared to the analogous cover $X_0(p)\to X(1)$ of modular curves corresponding to the level structure $\Gamma_0(p)$.

In the one-dimensional setting, $X_0(p)$ has field of definition $\Q(\zeta_p)$ and one only obtains a $\PGL_2(\F_p)$-cover. Under certain congruence conditions, Atkin-Lehner involutions can be used to turn this into a $\PSL_2(\F_p)$-cover by what has become known as Shih's trick \cite{Shih}. In our situation such a trick is unnecessary.
\end{remark}

\section{Arithmetic of algebraic cycles on Hilbert modular surfaces}\label{Sec:cycles}
In order to prove the Hilbert Property for a Hilbert modular surface $X_g$, it is crucial that we work on the minimal smooth, projective model $Y_g$ and have a sufficiently precise understanding of its algebraic cycles and their intersection behaviour, not only geometrically but over the base field $\Q$. 

There are three main sources of algebraic cycles: (i) resolution components of the cusp singularities; (ii) resolution components of the elliptic singularities; and (iii) Hirzebruch--Zagier divisors. We deal with each of these in turn. 

\paragraph{Cusp divisors.~}\label{SSec:cusps}
We follow \cite[Ch.~I-II, X.2]{VDG}. The boundary $X_g^\infty=X'_g\setminus X_g$ of the Baily--Borel compactification is a zero-dimensional $\Q$-scheme, which geometrically decomposes into $\#\Cl(K)$ points:
\[X_g^\infty(\ov\Q)=\bigsqcup_{[\mathfrak c]\in\Cl(K)}s_{[\mathfrak c]}\]

The classes in $\Cl(K)$ are in bijection with $\PSL(\Lambda)$-orbits of points in $\P^1(K)$ via the map sending $(x_0:x_1)\in\P^1(K)$ to the class $[x_0\cdot \mathfrak o + x_1\cdot \mathfrak a]\in\Cl(K)$. Fix a class $[\mathfrak c]\in\Cl(K)$ and write shorthand $s=s_{[\mathfrak c]}$. To $s$ we associate $M_s=\mathfrak c^{-2}\mathfrak a^{-1}$ viewed as an \emph{ordered} $\mathfrak o$-module, i.e.\ equipped with an order on $M_s\otimes_\sigma \R$ for each embedding $\sigma:K\hookrightarrow\R$. The isomorphism class of $M_s$ does not depend on the chosen representative $\mathfrak c$. %

\begin{proposition}\label{Prop:CoverQ}\hfill
 \begin{enumerate}
  \item The action of $\Gamma_\Q$ on $X_g^\infty(\ov\Q)$ is trivial, in other words all cusps are defined over $\Q$.
  \item The exceptional divisor $E_s$ of the minimal resolution $\tilde X_g\to X_g$ over any cusp $s$ is a cycle of smooth rational curves intersecting transversally.
 \end{enumerate}
\end{proposition}

By a cycle of smooth rational curves intersecting transversally we mean a connected curve whose irreducible components are isomorphic to $\P^1_\Q$, where each of them intersects exactly two other components, and it does so transversally. We call such a component in $\tilde X_g$ a \emph{cusp divisor}.

\begin{proof}
 \begin{enumerate}
  \item By \cite[\S1.3]{HLR} (cf. \cite[\S12]{Pink} for a more general treatment), the action of $\Gamma_\Q$ on $X_g^\infty(\ov\Q)$ factors through the natural action of $\operatorname{Gal}\left(\mathbb{Q}^{\mathrm{ab}} / \mathbb{Q}\right) \simeq \pi_0\left(\Q^{\times}  \backslash \mathbb{I}_\Q\right)$ on $\Cl(K)=K^{\times}  \backslash \mathbb{I}_K/\prod_v \mathfrak o_v^\times\cdot (\R^\times)^2$. This action factors through $\Cl(\Q)=1$ and is hence trivial.
  \item This essentially follows from the contents of \cite[\S1.6]{HLR} (see also \cite[Ch.~X]{VDG}), but we include the details for completeness. The set $(M_s\otimes \R)_+$ of totally positive elements in $(M_s\otimes \R)$ has a canonical decomposition as a smooth polyhedral fan $\Sigma$ given by the boundary points of the convex hull of totally positive elements $(M_s)_+\subset M_s$. (The fan $\Sigma$ can be computed via Hirzebruch--Jung continued fractions as explained in \cite[II.5]{VDG}.) The action of $\mathfrak o^{\times2}$ on $M$ induces a free action on the fan $\Sigma$.

    \vskip2mm

  Let $T=T(\Sigma)$ be the associated (infinite) toric scheme over $\Q$. This is a formally smooth scheme locally of finite type over the field of definition of $s$, which by (1) equals $\Q$. Let $T_{/\infty}$ be the completion of $T$ along its boundary $T^{\infty}$. Then by \cite[\S1.6]{HLR}, the completion of $\tilde X_g$ along $E_s$ is isomorphic to the quotient of $T_{/\infty}$ by ${\mathfrak o^{\times2}}$. Since $T(\Sigma)$ is glued from smooth affine toric varieties, each isomorphic to $\A^2_\Q$, this proves the claim. \qedhere
 \end{enumerate}
\end{proof}

\begin{minipage}{0.55\textwidth}
    \begin{tikzpicture}[scale=0.8]
        \draw[thick] (-2,0) -- (5,0);
        \draw[thick] (0,-2) -- (0,5);

        \filldraw (-1.0000, -1.0000) circle (2pt);
        \filldraw (0.61800, -1.6180) circle (2pt);
        \filldraw (-1.6180, 0.61800) circle (2pt);
        \filldraw (0.00000, 0.00000) circle (2pt);  %
        \filldraw (1.6180, -0.61800) circle (2pt);
        \filldraw (3.2360, -1.2360) circle (2pt);
        \filldraw (-0.61800, 1.6180) circle (2pt);
        \filldraw (1.0000, 1.0000) circle (2pt);    %
        \filldraw (2.6180, 0.38200) circle (2pt);   %
        \filldraw (4.2360, -0.23600) circle (2pt);
        \filldraw (-1.2360, 3.2360) circle (2pt);
        \filldraw (0.38200, 2.6180) circle (2pt);   %
        \filldraw (2.0000, 2.0000) circle (2pt);
        \filldraw (3.6180, 1.3820) circle (2pt);
        \filldraw (-0.23600, 4.2360) circle (2pt); 
        \filldraw (1.3820, 3.6180) circle (2pt);
        \filldraw (3.0000, 3.0000) circle (2pt);
        \filldraw (4.6180, 2.3820) circle (2pt);
        \filldraw (2.3820, 4.6180) circle (2pt);
        \filldraw (4.0000, 4.0000) circle (2pt);   

        \draw[thick] (5,0.331) -- (2.6180, 0.38200) -- (1.0000, 1.0000) -- (0.38200, 2.6180) -- (0.331,5);
        \draw[pattern={Lines[angle=45]}] (5, 0.38200*5/2.6180) -- (0,0) -- (5,5); 
        \draw (3,1.8) node[below]{$\sigma_1$};
        \draw[pattern={Lines[angle=-45]}] (5,5) -- (0,0) -- (0.38200*5/2.6180, 5); 
        \draw (1.8,3) node[left]{$\sigma_2$};

        \end{tikzpicture}
\end{minipage}%
\begin{minipage}{0.45\textwidth}
    \raggedleft
  
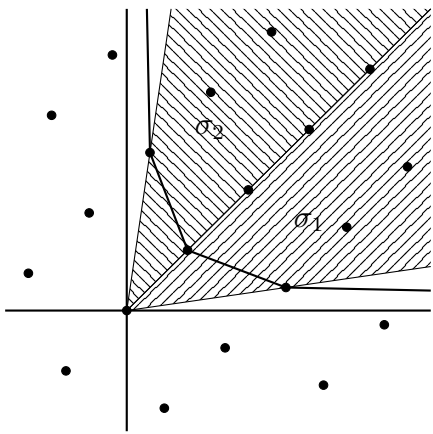
\captionof{figure}{The convex hull of $(M_s\otimes \R)_+$ and two of the (infinitely many) maximal cones of the fan $\Sigma$ highlighted, for $K=\Q(\sqrt 5)$, $s= $ principal cusp.}\label{fig:enter-label}
\end{minipage}

\vskip4mm

\paragraph{Elliptic divisors.~}
A point on a Hilbert modular surface $X_g(\Gamma)$ is said to be {\em elliptic} if its isotropy group in $\Gamma$ is non-trivial. We say that an elliptic point is of {\em type} $r$ if its isotropy group has order $r$. Assuming that $D \neq 5,8,12$, only points of type $2$ and $3$ appear. 
We further subdivide points of order $3$: we say that an elliptic point is of type $3^+$ if its rotation factor is $(\zeta_3,\zeta_3)$, and $3^-$ if its rotation factor is $(\zeta_3,\zeta_3^{-1})$.

\begin{proposition}\label{Prop:elliptic} Let $x\in X_g(\ov \Q)$ be a singularity of type $2$ or $3^\pm$. The minimal resolution of $x$ over $\ov\Q$ has the following exceptional divisors:
\begin{enumerate}
    \item type $2$: $\P^1_{\ov \Q}$ with self-intersection $-2$.
    \item type $3^+$: $\P^1_{\ov \Q}$ with self-intersection $-3$.
    \item type $3^-$: $\P^1_{\ov \Q}\cup \P^1_{\ov \Q}$, each component with self-intersection $-2$, meeting transversally in one point.
\end{enumerate}
\end{proposition}
\begin{proof}
    This is well-known, and we only recall the theory in the case of type $3^-$ here, as it will be used later. For a general treatment on the desingularisation of all quotient singularities on surfaces, see \cite[Sec.\ 2.6]{VDG}. For a more specific treatment of only the types of singularities that appear in this proposition, see \cite[Table, p.\ 65]{VDG}.

     We may choose local analytic coordinates $w_1,w_2$ near a lift $\tau\in(\H^+)^2$ of $x$ such that $\gamma$ acts as $(w_1,w_2) \mapsto (\zeta_3 w_1, \zeta_3^{-1}w_2)$. Let $U$ be a local $\gamma$-invariant neighbourhood of $\tau$. In the next paragraph we describe a neighbourhood $V$ of the exceptional divisor above $x$ in $\tilde X_g$.

We blow up $U$ three times, as follows. Start from the blowup $\Bl_{\tau}U$, where the (lift of) the action by $\gamma$ fixes only two points $\tilde \tau_1$ and $\tilde \tau_2$, both lying on the exceptional divisor $E_{\tau} \subset \Bl_{\tau}U$, and corresponding to the tangential directions $w_1=0$ and $w_2=0$. Consider the blowup $\Bl_{\tilde \tau_1,\tilde \tau_2,\tau}U$ of $\Bl_{\tau}U$ at these two points. The exceptional divisor $E_{\tau}+E_{\tilde \tau_1}+E_{\tilde \tau_2}$ of these triple blowup has three components with intersection table $E_{\tau}^2=-3, E_{\tilde \tau_1}\cdot E_{\tilde \tau_2}=0, E_{\tilde \tau_i}^2=-1, E_{\tilde \tau_i}\cdot E_{\tau}= 1, i=1,2$. The (lift of) the action by $\gamma$ fixes pointwise the two divisors $E_{\tilde \tau_i},i=1,2$, and preserves the divisor $E_{\tau}$, on which however it acts freely away from the two points $E_{\tau} \cap E_{\tilde \tau_i},i=1,2$. Then the quotient of $\Bl_{\tilde \tau_1,\tilde \tau_2,\tau}U$ by this action is a smooth analytic variety $V'$. We denote by $\pi:\Bl_{\tilde \tau_1,\tilde \tau_2,\tau}U \to V'$ the natural finite map. Let $E'_0, E'_1$ and $E'_2$ denote the divisors $\pi(E_{\tau}), \pi(E_{\tilde \tau_1})$ and $\pi(E_{\tilde \tau_2})$, respectively. We have $\pi^{-1}(E'_0)=E_\tau, \pi^{-1}(E'_i)=3E_{\tilde \tau_i}, i=1,2$, and we may use this to infer the intersection table of $E'_0, E'_1$ and $E'_2$ from that of the exceptional divisors in $\Bl_{\tilde \tau_1,\tilde \tau_2,\tau}U$. Explicitly, the matrix $(E'_i.E'_j)=\frac13(\pi^{-1}(E'_i).\pi^{-1}(E'_j))$ is equal to $\begin{pmatrix}
        -1 & 1 & 1\\
        1 & -3 & 0\\
        1 & 0 & -3
    \end{pmatrix}$ as shown in Figure~\ref{fig:3minus}.
\vskip4mm    
\begin{minipage}{0.55\textwidth}
    \begin{tikzpicture}[scale=.8]
  \draw (0,1) -- node[below] {$E_0'^2=-1$} (4,1)
        (.5,0) -- node[left] {$E_1'^2=-3$}(.5,3.5)
        (3.5,0) -- node[right] {$E_2'^2=-3$}(3.5,3.5);
\end{tikzpicture}
\end{minipage}%
\begin{minipage}{0.45\textwidth}
    \raggedleft
  \captionof{figure}{Exceptional divisors on the blowup $\Bl_{\tilde \tau_1,\tilde \tau_2,\tau}U$ of a $3^-$ elliptic point.}\label{fig:3minus}
\end{minipage}    
\vskip4mm
Since $(E'_0)^2=-1$ we may blow this divisor down and we denote the blow-down by $V$. On $V$ we only have two exceptional components $E_1$ and $E_2$ (the images of $E'_1$ and $E'_2$), with $E_1^2=E_2^2=-1$ and $(E_1.E_2)=1$.
\end{proof}

We call an exceptional divisor from Proposition~\ref{Prop:elliptic} on the minimal resolution $\tilde X_g$ of $X_g$ an \emph{elliptic divisor}.  

\vskip1mm

The following remark will not be used and only serves as clarification.

\begin{remark}
In the Shimura language, the elliptic points arise from  totally imaginary quadratic extensions $L=K(\alpha)$ of $K$ with an embedding of $K$-algebras $\iota:L \hookrightarrow M_2(K)$. 

The embedding $\iota$ determines an embedding $\iota:R_{L/\mathbb{Q}}\G_m\hookrightarrow R_{K/\mathbb{Q}}\GL_2$. Let $T_L$ be the torus defined as the product $R_{L/\mathbb{Q}} \G_m \times_{R_{K/\mathbb{Q}}\G_m} \G_m$, where the maps $R_{L/\mathbb{Q}} \to R_{K/\mathbb{Q}}\G_m$ and $\G_m \to R_{K/\mathbb{Q}}\G_m$ are the norm map $N_{L/K}$ and the natural embedding, respectively. The restriction of $\iota$ to $T_L$ determines an embedding $i:T_L \hookrightarrow G$.

To each CM-type of $L$, i.e.\ a subset $\Phi=\{\phi_1,\phi_2\}\subset\Hom(L,\ov\Q)$ with $\Hom(L,\ov\Q)=\Phi\sqcup\ov\Phi$, we may associate an embedding $h_\Phi:\BS\to T_L(\R)\cong_{\Phi} \{(z_1,z_2) \in \C^\times \times\C^\times: |z_1|^2=|z_2|^2\}, z\mapsto (z,z)$, where the isomorphism is via $\phi_1$ on the first factor and $\phi_2$ on the second. There always exists a CM-type of $L$ which extends $i$ to an embedding of Shimura data $(T_L,h_\Phi)\hookrightarrow(G,h)$. All elliptic points on $X_g$ arise as geometric connected components of such an embedding.

We note that in general the Galois action on the connected components is nontrivial and so the elliptic points are not defined over $\Q$ (unlike what one may understand from \cite[(1.5)]{Langer}, cf.\ \cite[XI.2.(1)]{VDG}).
\end{remark}

\paragraph{Hirzebruch--Zagier divisors.~}
For a matrix $B\in M_2(K)$, we denote its entrywise conjugate by $B'$. In the following, assume that $B$ is
\begin{enumerate}
    \item \emph{skew-hermitian}: $B'^t=-B$,
    \item \emph{integral for the genus $g$}: $B=\begin{pmatrix}
        a_1\sqrt{D}& \lambda\\
        -\lambda' & {a_2\sqrt{D}}/{A}
    \end{pmatrix}$
    for $a_1,a_2\in\Z$, $\lambda\in\mathfrak a^{-1}$,
    \item \emph{primitive}: $a_1,a_2,\lambda$ are not divisible by any same integer $>0$.
\end{enumerate} 

Define \[
    \H_B^{\pm} := \{(z_1,z_2) \in (\H^{\pm})^2 : (z_2,1)B(z_1,1)^t=0\},
    \]
    and $\H_B:= \H_B^+$. Let $F_B$ be the image of $\H_B$ in $X_g$, and $F_N:=\cup_B F_B$ where $B$ runs over all (skew-hermitian, integral, primitive) matrices with $\det(B)A=N$.

\begin{remark}\label{Rem:quaternionorder}
 In \cite[V.1.5(ii)]{VDG} a description of the level structure on $F_B$ is given. This description appears to not fully capture the level structure without additional assumptions on $B$. For example for $D=24$, $\Lambda=\mathfrak o\oplus\mathfrak o$, and $B={ \begin{pmatrix} 5\sqrt{D} & 57 \\ -57 & -27\sqrt{D}) \end{pmatrix}}$ the element $(57/\sqrt{D}+j)/5$ in the notation of \cite{VDG} does not lie in the correct quaternion order.
\end{remark}

\begin{theorem}\label{Thm:HZoverQ}
    For each $N \in \mathbb Z_{>0}$, the (possibly reducible) curve $F_N $ is defined over $\mathbb{Q}$.
\end{theorem}
\begin{proof}
    Let $\mathfrak M$ denote the $\Z$-module of skew-hermitian, integral matrices $B=\linebreak\begin{pmatrix}
        a\sqrt D & \lambda \\ -\lambda' & b\sqrt D/A
    \end{pmatrix}$ in $M_2(K)$ such that $a$ and $b$ are integers, and $\lambda$ lies in $\mathfrak a$. Let $B \in \mathfrak M_{\text{prim}}$ (recall that we denote by $M_{\text{prim}}$ the primitive elements of a module $M$) be such that $\det(B)=N/A$. We define the quaternion algebra
    \[
    Q_B := \{M\in M_2(K):M'^t B M = \det(M)B\},
    \]
    and let $\mathbf Q_B := \GL_1(Q_B)$ be the corresponding inner form of $\GL_2$. The natural embedding of multiplicative groups $Q_B^* \subset \GL_2(K)$ extends to an embedding of algebraic groups $\mathbf Q_B \subset R_{L/\mathbb Q}\GL_2$. The composition $Q_B^* \hookrightarrow \GL_2(K) \xrightarrow{\text{det}} L^*$ coincides with the usual reduced norm on $Q_B^*$ and has values in $\mathbb Q^*$. Analogously, working functorially, the composition $\mathbf Q_B \hookrightarrow R_{L/\mathbb Q}\GL_2 \xrightarrow{\text{det}} R_{L/\mathbb Q}\G_m$ is the reduced norm $\mathrm{Nrd}:\mathbf Q_B \to \G_m$. In particular, $\mathbf Q_B$ is contained in $G$. We denote by $\iota:\mathbf Q_B\hookrightarrow G$ the corresponding closed embedding. 
    
    The splitting $K \otimes_{\mathbb Q}\mathbb R=\mathbb R \oplus \mathbb R$ gives a natural identification
    \begin{equation}\label{EqQBR}
        Q_B(\mathbb R)=\{(M_1,M_2)\in \GL_2(\mathbb R)^2:M_2^t B M_1 = \det(M_1)B\}.
    \end{equation}
    It is clear from this identification that the first projection $pr_1:Q_B(\mathbb R) \to \GL_2(\mathbb R)$ is an isomorphism, and it extends to an isomorphism of algebraic group $pr_1:\mathbf Q_{B,\mathbb R} \to \GL_{2,\mathbb R}$.
    The natural (transitive) left action of $Q_B(\mathbb R)$ on $\H_B^+ \cup \H_B^-$ 
    identifies $\H_B^+ \cup \H_B^-$ with the conjugacy orbit of the homomorphism
    \[
    \BS \to \GL_{2,\mathbb R} \cong \mathbf{Q}_{B,\mathbb R},\ a + bi \mapsto \left(M_1(a,b)\coloneqq\begin{pmatrix} a & b \\ -b & a \end{pmatrix},M_2(a,b)\right),
    \]
    where $M_2(a,b)$ is uniquely determined by \eqref{EqQBR}.
    This identification makes the pair $(\mathbf Q_B,\H_B^+ \cup \H_B^-)$ into a Shimura datum, whose reflex field is $\mathbb Q$ since it is attached to a quaternion algebra over $\Q$ \cite[Example 12.4]{Milne}.
    
    Consider the morphism of Shimura data $(\mathbf Q_B, \H_B^+ \cup \H_B^-) \to (G, (\H^+)^2\cup (\H^-)^2)$ induced by the natural embeddings $\mathbf Q_B \hookrightarrow G$ and $\H_B^{\pm} \hookrightarrow (\H^{\pm})^2$.
    Letting $\cK_B:=\mathbf Q_B \cap \cK_{\mathfrak a}$, this induces a morphism of Shimura varieties
    \[
    Q_B(\mathbb Q) \backslash Q_B(\A_\Q^\fin) \times (\H_B^+ \cup \H_B^-) / \cK_B \hookrightarrow G(\mathbb Q) \backslash G(\A_\Q^\fin) \times ((\H^+)^2\cup (\H^-)^2) / \cK_{\mathfrak a},
    \]
    or, equivalently,
    \begin{equation}\label{Morphism}
       Q_B(\mathbb Q)_+ \backslash Q_B(\A_\Q^\fin) \times \H_B / \cK_B \hookrightarrow G(\mathbb Q)_+ \backslash G(\A_\Q^\fin) \times (\H^+)^2 / \cK_{\mathfrak a}=X_g.
    \end{equation}
    This morphism is defined over $\mathbb Q$, being a morphism of Shimura varieties arising from a morphism of two Shimura data with reflex field $\mathbb Q$ \cite[Remark 13.8]{Milne} \cite[Corollary 5.4]{Deligne}. It follows that the image of \eqref{Morphism} is defined over $\mathbb Q$. %
    
    We claim that the image of \eqref{Morphism} is contained in $F_N$. For a subset $S \subset G(\A_\Q^\fin) \times (\H^+)^2$, we denote by $[S]$ its image in $X_g$. The image of \eqref{Morphism} is $\cup_{q \in Q_B(\A_\Q^\fin)}[q \times \H_B]$. Let $q=(q_p)_{p \in \cP}$ be in $Q_B(\A_\Q^\fin)$, and write $q=g\gamma^{-1}$ with $g \in G(\mathbb Q)_+$ and $\gamma \in \cK_{\mathfrak a}$ (such $g$ and $\gamma$ may always be found by connectedness of $X_g$, see also \cite[Prop.~4.18]{Milne}).  We have
    \[
    [q \times \H_B]=g^{-1}\cdot [q \times \H_B]\cdot \gamma=[1 \times g^{-1}\H_B],
    \]
    and $g^{-1}\H_B=\H_{B'}$ with $B'=(g')^t\cdot B\cdot g \in \mathfrak M \otimes \mathbb Q$. Write $r$ for the rational number $\prod_{p \in \cP}p^{-v_p(q_p)}$, and let $B'':=rB'$. Note that $B'' \in \mathfrak M_{\text{prim}}$. In fact, we have
    \begin{equation}\label{EqMatrices}
        B''=rB'=r(g')^t\cdot B\cdot g=r\cdot (\gamma')^t (q')^t B q \gamma = r \det(q) \cdot (\gamma')^t  B \gamma = j \cdot (\gamma')^t  B \gamma,
    \end{equation}
    where $j := r \det(q) \in \mathbb I_{\mathbb Q}$ is an invertible integral idele, and $\gamma$ is an invertible integral adelic matrix in $\GL_2$. It follows that $B''$ lies in $(\mathfrak M \otimes \Q) \cap (\mathfrak M \otimes \A_{\Z}^\fin)_{\text{prim}}=\mathfrak M_{\text{prim}}$. Moreover, \eqref{EqMatrices} gives us that $v_p(\det(B''))=v_p(\det(B))$ for all $p \in \cP$, and hence $\det(B'')=\pm \det (B)$. Finally, note that $\det(B'')/\det(B)=r^2N_{K/\mathbb Q}(\det g)$ is positive since $g$ lies in $G(\mathbb Q)_+$. Thus  $\det(B'')=\det (B)$, and since this holds for all $q$, we infer that the image of \eqref{Morphism} is contained in $F_N$, proving the claim.

    The claim shows that the Galois orbit of each component $F_B$ of $F_N$ is contained in $F_N$, and concludes the proof.
\end{proof}

\begin{remark}
    The $Gal(\ov\Q/\Q)$-action on the geometrically irreducible components of $F_N$ is determined by the same theory as described in Subsection \ref{SSec:HMSQ}. In practice however, this is made difficult by the complicated level structure, see also Remark~\ref{Rem:quaternionorder}. Instead, one can often proceed as in Remark~\ref{Rem:HZoverQ}.
\end{remark}

The strict transform of $F_N$ along $\tilde X_g \dashrightarrow X_g$ is a curve (over $\Q$) on the minimal resolution $\tilde{X}_g$. From now on, by abuse of notation, we denote this curve on $\tilde X_g$ by $F_N$ and its geometrically irreducible components by $F_B$. We call the latter \emph{Hirzebruch--Zagier divisors}.

\subsection{Intersections of geometrically irreducible divisors on Hilbert Modular Surfaces of K3 type}
We recall the following classification of the Enriques--Kodaira type of $X_g$ due to Hirzebruch, Zagier and (for certain genera) van der Geer:

\begin{theorem}[{\cite[Theorems 2 and 3]{HZ}, \cite[Theorem VII.3.3]{VDG}}]\label{Thm:kodaira}
    The Hilbert modular surface $X_g$ has Kodaira dimension $0$ if and only if
    $$
        \begin{array}{ll}
        D\in \cS_+\coloneqq\{29,37,40,41,44,56,57,69,105\}, & g=\text { principal genus, or} \\
        D\in\cS_-\coloneqq\{21,24,28,33,40\}, & {g}=\text { nonprincipal genus} .
        \end{array}
    $$
    In all of these cases, $X_g$ is birational to a K3 surface.
\end{theorem}

Furthermore, in the above cases the minimal model $Y_g$ is obtained from $\tilde X_g$ by first contracting all components of $F_N$ for $N=1,2,3,4$ and $N=9$ when $3 \mid D$, and then successively any elliptic divisors whose self-intersection becomes $-1$ \cite[\S VII.4]{VDG}.

The main theorem of this subsection is the following:

\begin{theorem}\label{Thm:intersection}
    The intersection graphs of the sets of cuspidal, elliptic, and Hirzebruch--Zagier divisors with $N \leq 10$ on the minimal Hilbert modular surfaces $Y_g(D)$ of K3 type are the ones listed in Appendix~\ref{Sec:graphs}.
\end{theorem}

\paragraph{Intersection of Hirzebruch--Zagier and cusp divisors.~}The intersection of $F_N$ with the cusp divisors was explicitly obtained in \cite[\S V.5]{VDG}. Analysing the argument more precisely, we now give a description of the intersection of the components $F_B$ with the cusp divisors.

Recall from \S\ref{SSec:cusps} the ordered $\mathfrak o$-module $M_s$ associated to a cusp $s$ and its canonical fan decomposition $\Sigma$, which gives rise to the toroidal resolution of $s$. Let $\sigma\in\Sigma$ be a cone spanned by boundary points $x_{\sigma,1}, x_{\sigma,2}\in (M_s)_+$. The ray through $x_{\sigma,i}$ corresponds to a cusp divisor $C_{\sigma,i}$ at the boundary $Z(\sigma)^\infty$ of the affine toric variety $Z(\sigma)$. In a formal neighbourhood, this divisor is given by the vanishing of the coordinate axis $u_{\sigma,i}=0$.

We make use of the fact that the cusp $s_{[\mathfrak o]}$ at infinity can be translated to any cusp $s_{[\mathfrak c]}$ by a suitable element $t_s\in\SL_2(K)$. (This translation realises an isomorphism $X_{[\mathfrak a\mathfrak c^2]}\cong X_{[\mathfrak a]}$.) 

\begin{proposition}\label{Prop:HZintCusp}
    Let $s_{[\mathfrak c]}$ be a cusp with norm $C=\Nm(\mathfrak c)$ and let $\sigma$ be a cone in the canonical decomposition of $M_s$.

    Define \[B_{p,q,b}=\begin{pmatrix}
        0 & \lambda_{p,q}\\
        -\lambda_{p,q}' & \frac{b}{AC^2}\sqrt{D}
    \end{pmatrix}, \lambda_{p,q}=p x_{\sigma,1}+qx_{\sigma,2}.\]

    In a formal neighbourhood of the boundary $Z(\sigma)^\infty$, the vanishing locus of
    \[u_{\sigma,1}^{p/n} = \zeta_n^b u_{\sigma,2}^{q/n},\]
    such that $p,q$ are non-negative integers and $\zeta_n=e^{2\pi i/n}$ where $n=\gcd(p,q)$, is a branch of the curve $F_{t_s'^t B_{p,q,b} t_s}$. All branches of Hirzebruch--Zagier divisors intersecting the cusp divisors over $s$ are of this form.
\end{proposition}
\begin{proof}
    Let $B\in M_2(K)$ be a primitive, skew-hermitian matrix and assume that $F_B$ passes through $s$. After translating by $t_s^{-1}$, which has the effect of replacing $B$ by ${t_s^{-1}}'^tBt_s^{-1}$ and $\mathfrak a$ by $\mathfrak a\mathfrak c^2$, we may assume that $s=s_{[\mathfrak o]}$.

    Now a cusp $x\in\P^1(K)$ lies on the preimage of $F_B$ in $(\H^+)^2$, if and only if $B$ can be chosen in its $\SL(\Lambda)$-orbits such that
    \begin{equation}\label{Eq:HZ}
        a\sqrt{D}xx'-\lambda' x + \lambda x'+\frac{b}{AC^2}\sqrt{D}=0
    \end{equation}
    For the cusp $(1:0)$ at infinity this means that $a=0$.

    By \cite[II.2]{VDG}, the coordinates $z_1,z_2$ and $u_{\sigma,1},u_{\sigma,2}$ can be chosen to satisfy
    \[2\pi i z_j=x_{\sigma,1}^{(j)}\log u_{\sigma,1}+x_{\sigma,2}^{(j)}\log u_{\sigma,2},\quad j=1,2,\]
    where $x^{(j)}$ denotes the image of $x\in K$ under the $j$-th real embedding.

    Writing $\lambda$ as $\lambda_{p,q}$ in the basis $x_{\sigma,1},x_{\sigma,2}$ and plugging into (\ref{Eq:HZ}), one obtains
    \[(-p\log u_{\sigma,1} + q\log u_{\sigma,2})\left(x_{\sigma,1}^{(1)}x_{\sigma,2}^{(2)}-x_{\sigma,1}^{(2)}x_{\sigma,2}^{(1)}\right)+b\sqrt{D}\Nm(\mathfrak a\mathfrak c^2)^{-1}=0\]
    which becomes
    \[u_{\sigma,1}^{p/n}=\zeta_n^{b/n}u_{\sigma,2}^{q/n}.\]
\end{proof}
\begin{remark}\label{Rem:HZoverQ}
    If Proposition~\ref{Prop:HZintCusp} gives that any two components of $F_N$ intersect the cusp resolution in different curves, one may conclude that all components of $F_N$ are defined over $\Q$. Indeed this follows immediately from Proposition~\ref{Prop:CoverQ} and  Theorem~\ref{Thm:HZoverQ}.
\end{remark}

\paragraph{Intersection of Hirzebruch--Zagier and elliptic divisors.~}We start with the following observation (which may also be found e.g.\ in the discussion after Definition 6.1 in \cite{VDG} without a proof):

\begin{lemma}\label{Lem:transversal}
    For nonsimilar $B,B' \in M_2(K)$, the intersection number  $(\H_B.\H_{B'})_{(\H^+)^2}$ is $\leq 1$.
\end{lemma}
\begin{proof}
    In fact, both $\H_B$ and $\H_{B'}$ are the intersection of graphs in $\P^1(\C)^2$ of linear fractional transformations with real coefficients with the open $(\H^+)^2\subset \P^1(\C)^2$. Two such graphs intersect in $\P^1(\C)^2$ in exactly two (with multiplicity) complex-conjugate points, and at most one of them can have positive imaginary parts, or, equivalently, lie in $(\H^+)^2$.
\end{proof}

\begin{proposition}\label{Prop:HZEllmult}
The intersection number of a branch of a Hirzebruch--Zagier divisor and an elliptic divisor associated to an elliptic point of order $2$ or $3$ in $\tilde X_g$ is at most one.
\end{proposition}
\begin{proof}
    Let $x \in X_g(\C)=\Gamma \backslash (\H^+)^2$ be an elliptic point of order $r$, $\tau \in (\H^+)^2$ be a lift, and $F_B \subset \tilde X_g$ be a Hirzebruch--Zagier divisor that intersects the exceptional divisor above $x$. After possibly replacing $B$ with a matrix $T'^{t} B T, T \in \Gamma$ representing the same Hirzebruch--Zagier divisor, we may assume that $\tau \in \H_B$.

    \vskip2mm

    Assume first that {\bf $x$ is of type $2$ or $3^+$}. Let $\gamma \in \Gamma$ be the isotropy generator of $\tau$ whose differential $(d \gamma)_{\tau}$ at $\tau$ is $\zeta_r\cdot \operatorname{Id}$. Both $\H_B$ and $\gamma \cdot \H_B=\H_{\gamma'^{t} B \gamma}$ pass through $\tau$, and they are equal to the first order near this point, hence they are tangent to each other, and the observation above then tells us that $\gamma \cdot \H_B=\H_B$, i.e.\ $\H_B$ is invariant under $\gamma$.
    
    Since the differential $(d \gamma)_{\tau}$ is of the form $\zeta_r\cdot \operatorname{Id}$, a local neighbourhood of the exceptional divisor $E_{x}$ above $x$ in $\tilde X_g(\C)$ is obtained by taking the quotient the blowup $\Bl_{\tau}U$ in $\tau$ of a $\gamma$-invariant neighbourhood $U$ of $\tau$ in $(\H^+)^2$ by the lift of the action by the isotropy group of $\gamma$. Let $E_{\tau}$ (resp.\ $E_x$) be the exceptional divisor in $\Bl_{\tau}U$ (resp.\ in $\tilde X_g$), $\tilde \H_B$ be the proper transform of $\H_B$ in $\Bl_{\tau}U$. Let $V \coloneqq (\Bl_{\tau}U)/\langle \gamma \rangle$, and $\pi:\Bl_{\tau}U \to V$ be the natural projection. Note that $\pi$ is a finite, hence proper map, that $\tilde \H_B=\pi^{-1}(F_B)$ and that $E_x=\pi(E_{\tau})$. Then (a local version of) the projection formula gives $1=(\tilde \H_B.E_{\tau})=(F_B.E_x)$, as wished.

    \vskip2mm
    
    For the remaining case, assume that {\bf $x$ is of type $3^-$}. 

    We use the notation from the proof of Proposition~\ref{Prop:elliptic}. Consider the intersections of $\H_B,\gamma \H_B$ and $\gamma^2\H_B$ near $\tau$. If the tangential direction of $\H_B$ in $\tau$ is not $w_1=0$ or $w_2=0$, then these three curves intersect pairwise transversally at $\tau$. If the tangential direction of $\H_B$ in $\tau$ is $w_1=0$ or $w_2=0$, then the three curves intersect tangentially and so they must be equal by the observation at the very beginning of this proof. We treat the two cases separately.

    Let $\tilde \H_B$ be the proper transform of $\H_B$ in $\Bl_{\tilde \tau_1,\tilde \tau_2,\tau}U$, and $F_B'$ be the proper transform of $F_B$ in $V'$. In the first case, $\tilde \H_B$ intersects transversally $E_{\tau}$, it does not intersect $E_{\tilde \tau_i},i=1,2$, and we have $\pi^{-1}(F_B')=\tilde\H_B + \gamma \tilde\H_B+\gamma^2 \tilde\H_B$. Then the projection formula gives $(3E'_0, F_B')=(E_{\tau},\tilde\H_B + \gamma \tilde\H_B+\gamma^2 \tilde\H_B)$ (even though $\pi(E_{\tau})=E'_0$ {\em scheme-theoretically}, we have $\pi_*(E_{\tau})=3E'_0$ in the Chow groups), and the latter is $3$, hence $(E'_0, F_B')=1$. Analogously, we get $(E'_i, F_B')=(E_{\tilde\tau_i},\tilde\H_B + \gamma \tilde\H_B+\gamma^2 \tilde\H_B)=0$ for $i=1,2$. After blowing down $E'_0$, we get $(F_B,E_i)=1$ for $i=1,2$, as wished. In the second case, we may assume without loss of generality that $\H_B$ is tangent to $w_1=0$. The proper transform of $\H_B$ in $\Bl_\tau U$ then intersects $E_{\tau}$ only in the point $\tilde \tau_1$, and it does so transversally, and hence the curve $\tilde \H_B$ in $\Bl_{\tilde \tau_1,\tilde \tau_2,\tau}U$ does not intersect $E_{\tau}$ and $E_{\tilde \tau_2}$ and it only intersects $E_{\tilde \tau_1}$ in one point, and it does so transversally. Moreover, in this second case we have $\pi^{-1}(F_B)=\tilde \H_B$, and then the projection formula gives $(E'_0,F_B')=\frac13(3E'_0, F_B')=(E_{\tau},\tilde \H_B)=0$, and $(E'_i, F_B') = (E_{\tilde \tau_i}, \tilde \H_B)$ which is $1$ for $i=1$ and $0$ for $i=2$. In the blow-down $V$, we then get $(E_1,F_B)=1$ and $(E_2,F_B)=0$, concluding the proof.
\end{proof}

It remains to determine which curves $F_B$ meet which elliptic divisors. Rather than determining the elliptic points on the image of $F_B$ in $X_g$ for fixed $B$, it is easier to fix the elliptic point and compute a list of suitable $B$.
Let $\tau=(z_1,z_2)\in(\H^+)^2$ be the lift of an elliptic point of order $2$ or $3$. The free $\Z$-submodule $L_\tau\subset\mathfrak M$ of skew-hermitian, integral matrices $B$ such that $(z_2, 1)B(z_1, 1)^t=0$ has rank $2$ (see \cite[Def.~6.1, Prop.~7.1]{VDG}). Writing $B_1, B_2$ for a basis of $L_\tau$, define the positive definite integral binary quadratic form
\[\phi_\tau(x,y)=\det(xB_1+yB_2)\cdot A.\]
One immediately obtains the following.

\begin{proposition}\label{Prop:HZintEll}
    Let $x$ be an elliptic point of order $2$ or $3$ and $\tau\in(\H^+)^2$ a lift. The mapping $(x,y)\mapsto xB_1+yB_2$ defines a bijection between integral solutions $\phi_\tau(x,y)=N$ and branches of $F_N$ meeting an elliptic divisor over $x$.
\end{proposition}

\paragraph{Intersections of Hirzebruch--Zagier divisors.~}Define $T_N=\bigcup_{d^2|N}F_{N/d^2}$ and $T'_N$ as the image of $T_N$ in the singular variety $X_g$.

Since $X_g$ has only quotient singularities, one can, for positive integers $M$ and $N$, define the transversal intersection number $(T_M',T_N')_\mathrm{tr}$ as follows: For a point $x\in X_g$ with isotropy group $\Gamma_x$, fix a lift $\tau\in(\H^+)^2$ and let 
$\mathfrak{I}_x$ be the set of pairs of distinct branches $(\H_B,\H_{B'})$, $B,B'\in\mathfrak{M}_\mathrm{prim}$, such that $\H_B$ (resp.\ $\H_{B'}$) maps to $T_M'$ (resp.\ $T_N'$). Then set $(T_M',T_N')_x=\#\mathfrak{I}_x/\#\Gamma_x$ and $(T_M',T_N')_\mathrm{tr}=\sum_{x\in X_g}(T_M',T_N')_x$.

Let
\begin{align*}
H(n)=&\sum_{d^2|n}h'(-n/d^2),\, n>0\\
H_D^o(n)=&\sum_{\substack{s\in \Z\\ s^2< 4n\\ s^2\equiv 4n\pmod D}} H\left(\frac{4n-s^2}{D}\right)
\end{align*}
and let $h'(n)$ be the weighted class number, i.e. $h(-3)=1/3$, $h(1/4)=1/2$, or otherwise equal to $h(\Q(\sqrt n))$.
Furthermore, for a prime $p|D$, set $P_p=M$ if $v_p(M)\leq v_p(N)$ and $P_p=N$ otherwise.

We recall the following formula, due to Hirzebruch and Zagier when $D$ is prime and $g$ is principal, and Hausmann for the general case:
\begin{proposition}[{\cite[(50)]{HZint}, \cite[(5.7) Satz]{Hausmann}}]\label{Thm:HZintHZ}
 Let $M,N$ be positive integers. Then 
 \[(T_M',T_N')_\mathrm{tr}=\frac{1}{2}\sum_{d|\gcd(M,N)}dH_D^o(MN/d^2)\prod_{p|D}\left(\legendre{D_p}{d}+\legendre{D_p}{P_pA/d}\right)\]
 where $D=\prod_{p|D}D_p$ is a factorisation of $D$ into prime power fundamental discriminants.
\end{proposition}

\begin{lemma}\label{Lem:HZcyc}
Let $x\in X_g$ be an elliptic point of order $2$ or $3$ with isotropy group $\Gamma_x$. Assume that a pair of distinct branches given by the images of $\H_B,\H_B'$, $B,B'\in\mathfrak{M}_\mathrm{prim}$, meet in $x$. Then the intersection number of their strict transforms on the elliptic divisors in $\tilde X_g$ above $x$ is equal to
\begin{enumerate}
    \item $1$, if $x$ is of type $3^-$ and neither $\H_B$ nor $\H_B'$ is preserved by the action of $\Gamma_x$ on $(\H^+)^2$.
    \item $0$, otherwise.
\end{enumerate}
\end{lemma}
\begin{proof}
    If $x$ is of type $2$ or $3^+$, this follows directly from Lemma~\ref{Lem:transversal} after blowing up in $x$. Assume that $x$ is of type $3^-$. If one of the branches is preserved by $\Gamma_z$, i.e. tangential to $w_1=0$ or $w_2=0$ as in the proof of Proposition~\ref{Prop:elliptic}, that branch will intersect the elliptic divisors above $x$ in either $E_1$ or $E_2$ but not both. If neither branch is preserved by $\Gamma_x$, then they intersect the exceptional divisor of Figure~\ref{fig:3minus} in the interior of $E_0'$, which after contraction of $E_0'$ becomes a single point of intersection.
\end{proof}

A recursive application of Proposition~\ref{Thm:HZintHZ} and Lemma~\ref{Lem:HZcyc} determines the transversal intersections of $F_M$ and $F_N$.

\begin{remark}
We stop short of giving a component-wise analysis of the intersections of Hirzebruch--Zagier divisors in non-elliptic special points. This is because the small divisors we consider ($F_N,\, N\leq 10$) turn out to have intersections supported in the elliptic points, which are already accounted for. 
\end{remark}

\paragraph{Proof of Theorem~\ref{Thm:intersection}}
One first computes the list of cusp, elliptic and Hirzebruch--Zagier divisors. Details of the computation are listed in the tables of Appendix~\ref{Sec:computations}. The intersections between these divisors are obtained by applying Propositions~\ref{Prop:HZintCusp}, \ref{Prop:HZEllmult}, \ref{Prop:HZintEll} and \ref{Thm:HZintHZ}.\qed

\section{Hilbert Property for Hilbert modular surfaces}\label{Sec:main}
We prove the following finer version of Theorem~\ref{Thm:K3}:
\begin{theorem}\label{Thm:K3regular}
     Let $X/\Q$ be the Hilbert modular surface of discriminant $D\in \cS_+$ (resp.\ $\cS_-$) and principal genus (resp.\ non-principal genus) at base level. Then the natural image of $X(\Q(t))$ in $X_{\bar \Q(t)}$ is non-thin.
\end{theorem}

\begin{proof}[Proof of Theorem~\ref{Thm:K3regular}]
    We apply Corollary~\ref{Cor:Hilbertellpticregular} to the surfaces $Y_g$ from \Cref{Thm:K3}. Using the results of \S\ref{Sec:cycles}, we find $(-2)$-curves and cuspidal elliptic fibres among the cusp divisors, elliptic divisors and Hirzebruch--Zagier divisors. The general outline of the argument is the same for all $D$ and $g$ in the theorem as follows (see Table~\ref{case} for each individual case).
    
    Let $\Gamma$ be the intersection graph of these curves, see Appendix~\ref{Sec:graphs}. Looking at the elliptic configurations in the graph, one infers that none of the divisors displayed, except possibly the ones in bold, are contained in the over-exceptional divisor $Z$.
    
    We exhibit a simply connected elliptic configuration $G \leq \Gamma$, none of whose components lies in $Z$. This gives a genus $1$ fibration $|G|:Y\to\P^1_\Q$. We then exhibit a second elliptic configuration $G'\leq \Gamma$ containing a section of $|G|$ (this section is listed as the first component in Table~\ref{case}). No $(-2)$-curve intersecting $G'$ is contained in $Z$, except possibly $F_5^1$ and $F_5^2$ in the case $D=44$: this is clear for curves that intersect $G'$ properly, and at the same time no component of $G'$ is contained in $Z$, after the exceptions made.

    We then exhibit two $(-2)$-curves $\sigma_0,\sigma_1$, which are sections of $|G|$ for all cases except $D=21$, where they are instead sections of $|G'|$. In every case, their difference is non-torsion, since either the fibration admits a fiber of type $II$ or the two sections intersect. Hence the fibration is elliptic and has Mordell--Weil rank $\geq 1$, and the non-constant elements in $Y(\Q(t))$ are Zariski--dense in $Y_{\Q(t)}$. 

    The complement $X_{\bar k}\setminus Z$ is simply connected by Lemma \ref{Lem8} applied to the fibration $|G|$. We need to verify the hypotheses. While (i) clearly holds, for (ii) we take the following component of a fiber of $|G|$.
    \begin{itemize}
        \item For $D \neq 44$, or for $D=44$ when the fiber of $|G|$ does not contain $F_5^1$ or $F_5^2$,
        we take any fiber component intersecting the section of $|G|$ contained in $G'$: this component is not contained in $Z$ as it intersects $G'$, and it has multiplicity $1$ as it is a fiber component intersecting a section.
        \item For $D=44$ and the fiber $\{F_5^1,E_3,F_9,E_5\}$ (the one containing $F_5^1$), we take the component $F_9$.
        \item The curve $F_5^2$ is not vertical for $|G|$ for $D=44$.
    \end{itemize}
    All assumptions of Corollary~\ref{Cor:Hilbertellpticregular} now hold, concluding the proof.
\end{proof}
    
\begin{remark}
    The Hilbert Property for all rational Hilbert modular surfaces (as listed in \cite[Theorem V.3.3]{VDG}) remains open. Those for which explicit models were computed in \cite{EK} are $\Q$-rational and so classical Hilbert Irreducibility applies. We do not know whether all rational Hilbert modular surfaces are $\Q$-rational.
\end{remark}

\vspace{-1mm}

\begin{proof}[Proof of Theorem~\ref{Thm:IGP}]
Elkies and Kumar \cite{EK} show that the surface $Y_g(D)$ is $\Q$-rational for $D=5,8,13,17$ and $g$ principal (there is only one genus for these discriminants) and $D=12$ and $g$ non-principal. For such a $D$, Theorem \ref{Thm:cover} applied to the $\Q$-rational modular surface $Y_g(D)$ gives a geometrically integral $\PSL_2(\F_{p^2})$-cover $\pi:Z \to \P^2_\Q$. By Bertini's theorem, the specialisation of $\pi$ over a line $l$ stays geometrically irreducible for all $l$ in some nonempty Zariski--open $U$ of the dual space $(\P^2_\Q)^*$. Any such specialisation provides a regular $\PSL_2(\F_{p^2})$-representation.

Now for $D \in \cS_+\cup \cS_-$, we get a regular $\PSL_2(\F_{p^2})$-representation for all $p$ with $\legendre{D}{p}=-1$ by combining Theorem \ref{Thm:cover} and Theorem \ref{Thm:K3regular}.

Combining the above, we proved the theorem for all $p$ for which there exists $D_p$ in
\[
\cD \coloneqq \{5,8,12,13,17,21,24,28,29,33,37,40,41,44,56,57,69,105\}
\]
with $\legendre{D_p}{p}=-1$. One easily checks by hand that this includes all primes $p \leq 41$. For $p > 41$, such a $D_p$ exists unless
\begin{center}
    $\legendre{D}{p}=1$ for all $D \in \cD$,
\end{center}
or, equivalently, unless $\legendre{D}{p}=1$ for all $D$ in the multiplicative subgroup $\langle \cD \rangle \subset \Q^*$ generated by $\cD$. Rewriting $\cD$ as
\[
\{5,2^3,2^2 \cdot 3,13,17,3 \cdot 7,2^3 \cdot 3,2^2 \cdot 7,29,3 \cdot 11,37,2^3 \cdot 5,41,2^2 \cdot 11,2^3 \cdot 7,3\cdot 19,3 \cdot 23,3 \cdot 5 \cdot 7\}
\]
one sees that $\langle \cD \rangle=\langle \cL \rangle$, where $\cL=\{\ell \text{ prime}: \ell \leq 41, \ell \neq 31\}$ is the list of prime factors appearing in $\cD$. So $D_p$ exists unless $\legendre{\ell}{p}=1$ for all $\ell$, concluding the proof.
\end{proof}

\vspace{-.2cm}
\small
\begin{table}[h!]
\caption{Genus $1$ fibrations in the proof of Theorem~\ref{Thm:K3regular}.}
\label{case}
\begin{tabular}{ |Sr|Sl|Sc|}
\hline 
\multicolumn{1}{|c|}{$D$} & \multicolumn{1}{c|}{$G,G'$} & $\sigma_0,\sigma_1$ \\
\hline\hline
\multicolumn{3}{|c|}{non-principal $g$}\\
\hline
$21$ & \makecell{$\{E_3 ,E_4, F_5^1,C_2,C_1,E_1^+,C_6\}$, $\{F_5^2,C_5,C_6\}$} & $C_1$, $C_4$ \\
\hline
$24$ & \makecell{$\{F_5, C_6, C_8, E_3, E_4\}$, $\{C_1, C_5\}$} & $C_1$, $C_5$ \\
\hline
$28$ & \makecell{$\{F_6,C_3, C_5, C_8, C_{10}\}$, $\{C_1, C_2, C_6, C_7, E_1, E_2^+\}$} & $C_1$, $C_7$ \\
\hline
$33$ & \makecell{$\{F_6, C_3, C_6, C_9, C_{12}\}$, $\{C_1, C_8\}$} & $C_1$, $C_8$ \\
\hline
$40$ & \makecell{$\{F_{10}, E_2, E_4, F_8, E_2^+\}$,$\{C_2, C_2'\}$} & $C_2$, $C_2'$ \\
\hline \hline 
\multicolumn{3}{|c|}{principal $g$}\\
\hline
$29$ & \makecell{$\{C_1\}$, $\{E_2^+, E_3^+, F_6\}$} & $E_2^+$, $E_3^+$ \\
\hline
$37$ & \makecell{$\{C_7\}$,$\{C_1, C_6, E_3^+\}$} & $C_1, C_6$ \\
\hline
$40$ & \makecell{$\{C_1\}$, $\{C_2, C_2', C_6, C_6', F_6\}$} & $C_2$, $C_6$ \\
\hline
$41$ & \makecell{$\{C_{11}\}$, $\{C_1, F_8, C_{10}, F_{10}\}$} & $C_1, C_{10}$ \\
\hline
$44$ & \makecell{$\{C_1\}$, $\{C_2, C_3, F_5^2, C_5, C_6, F_5^1\}$} & $C_2, C_6$ \\
\hline
$56$ & \makecell{$\{C_2\}$, $\{C_1, C_3\}$} & $C_1, C_3$ \\
\hline
$57$ & \makecell{$\{C_5\}$, $\{C_4, F_6, C_6, C_7, C_8, C_2, C_3\}$} & $C_4, C_6$ \\
\hline
$69$ & \makecell{$\{C_2\}$, $\{E_8^+,E_5,E_9^+,E_6,F_6\}$} & $C_1$, $C_3$ \\
\hline 
$105$ & \makecell{$\{C_6\}$, $\{C_1,C_2,C_4,C_5\}$} & $C_1$, $C_5$ \\
\hline 
\end{tabular}
\end{table}
\normalsize

\pagebreak
\appendix
\captionsetup[subfigure]{labelformat=empty}
\section{Intersection diagrams of divisors on $Y_g$}
\label{Sec:graphs}
The diagrams are computed with the help of Appendix~\ref{Sec:computations} using the results of \S\ref{Sec:cycles}.

A square box indicates that the divisor is an irreducible elliptic configuration. All other divisor are smooth rational. The naming conventions are:
\begin{itemize}
\item A cycle of cusp divisors\footnote{\label{nitpicking}strictly speaking, a curve which is the image of such a divisor under the blowdown morphism $\tilde X_g\to Y_g$.} is labelled $C_1,C_2,\dots$. If there is a second cusp, the divisors in its resolution cycle are apostrophised.
\item Elliptic divisors\textsuperscript{\ref{nitpicking}} of type $2$ are labelled $E_1,E_2,\dots$. Elliptic divisors of type $3^+$ are labelled $E_1^+,E
_2^+,\ldots$. Elliptic divisors of type $3^-$ are labelled as $E_1^{-},E_1^{-'},E_2^{-},E_2^{-'},\ldots$
\item The geometrically irreducible components of the Hirzebruch--Zagier divisors\textsuperscript{\ref{nitpicking}} $F_N$, $N\leq 10$, are labelled $F_N^1, F_N^2, \dots$.
\end{itemize}

Elliptic divisors that do not intersect the cusp and Hirzebruch--Zagier divisors (with $N \leq 10$) nor any other elliptic divisor (except for elliptic divisors in the same resolution of type $3^-$) are omitted. Divisors displayed in bold may lie in the over-exceptional divisor (see proof of Theorem~\ref{Thm:K3regular}). Code to reproduce these diagrams is available at \href{https://github.com/dgvirtz/hilbmodcycles}{github.com/dgvirtz/hilbmodcycles}.

\begin{figure}[H]
\centering
\begin{subfigure}{.45\textwidth}
\centering
\[\begin{tikzcd}[cramped, sep=tiny]
	& {C_2} & {E_3} & {C_1} & {C_6} \\
	{F_5^1} && {F_6} & {E_1^+} && {F_5^2} \\
	& {C_3} & {E_4} & {C_4} & {C_5}
	\arrow[curve={height=-12pt}, no head, from=1-2, to=1-4]
	\arrow[no head, from=1-2, to=3-2]
	\arrow[no head, from=1-3, to=2-1]
	\arrow[no head, from=1-3, to=2-6]
	\arrow[no head, from=1-4, to=1-5]
	\arrow[curve={height=-12pt}, no head, from=1-4, to=3-4]
	\arrow[no head, from=1-5, to=2-6]
	\arrow[no head, from=1-5, to=3-5]
	\arrow[no head, from=2-1, to=1-2]
	\arrow[no head, from=2-1, to=3-2]
	\arrow[no head, from=2-3, to=1-3]
	\arrow[equals, from=2-3, to=2-4]
	\arrow[no head, from=2-3, to=3-3]
	\arrow[no head, from=2-4, to=1-4]
	\arrow[no head, from=2-4, to=3-4]
	\arrow[no head, from=2-6, to=3-5]
	\arrow[curve={height=12pt}, no head, from=3-2, to=3-4]
	\arrow[no head, from=3-3, to=2-1]
	\arrow[no head, from=3-3, to=2-6]
	\arrow[no head, from=3-5, to=3-4]
\end{tikzcd}\]
\caption{$D=21$ (non-principal)}
\end{subfigure}
\hfill
\begin{subfigure}{.45\textwidth}
\centering
\[\begin{tikzcd}[cramped, sep=tiny]
	& {C_2} & {C_1} && {C_5} & {C_6} \\
	{C_3} & {F_5^2} & {E_3} & {F_8} & {E_4} & {F_5^1} & {C_7} \\
	& {C_4} & {E_1} & {F_6} & {E_2} & {C_8} \\
	&& {E_3^-} && {{E_3^-}'}
	\arrow[no head, from=1-2, to=2-1]
	\arrow[no head, from=1-3, to=1-2]
	\arrow[equals, from=1-3, to=1-5]
	\arrow[no head, from=1-3, to=2-4]
	\arrow[no head, from=1-3, to=3-6]
	\arrow[no head, from=1-5, to=2-4]
	\arrow[no head, from=1-6, to=1-5]
	\arrow[no head, from=2-1, to=3-2]
	\arrow[no head, from=2-2, to=1-2]
	\arrow[no head, from=2-2, to=2-3]
	\arrow[curve={height=12pt}, no head, from=2-2, to=2-5]
	\arrow[no head, from=2-3, to=2-4]
	\arrow[curve={height=12pt}, no head, from=2-3, to=2-6]
	\arrow[no head, from=2-4, to=4-5]
	\arrow[no head, from=2-5, to=2-4]
	\arrow[no head, from=2-5, to=2-6]
	\arrow[no head, from=2-6, to=1-6]
	\arrow[no head, from=2-6, to=3-6]
	\arrow[no head, from=2-7, to=1-6]
	\arrow[no head, from=3-2, to=1-5]
	\arrow[no head, from=3-2, to=2-2]
	\arrow[no head, from=3-4, to=2-1]
	\arrow[no head, from=3-4, to=2-4]
	\arrow[no head, from=3-4, to=2-7]
	\arrow[no head, from=3-4, to=3-3]
	\arrow[no head, from=3-4, to=3-5]
	\arrow[no head, from=3-6, to=2-7]
	\arrow[no head, from=4-3, to=2-4]
	\arrow[no head, from=4-3, to=4-5]
\end{tikzcd}\]
\caption{$D=24$ (non-principal)}
\end{subfigure}
\begin{subfigure}{.45\textwidth}
\centering
\[\begin{tikzcd}[cramped,sep=tiny]
	& {C_{10}} & {C_1} && {C_2} & {C_3} \\
	{C_9} & {F_7^2} & {E_2^+} & {F_6} & {E_1^+} & {F_7^1} & {C_4} \\
	{E_1} & {C_8} & {C_7} & {F_{10}} & {C_6} & {C_5} & {E_2} \\
	&&& {E_3}
	\arrow[no head, from=1-2, to=1-3]
	\arrow[no head, from=1-3, to=1-5]
	\arrow[curve={height=12pt}, no head, from=1-3, to=3-3]
	\arrow[no head, from=1-5, to=1-6]
	\arrow[curve={height=-12pt}, no head, from=1-5, to=3-5]
	\arrow[no head, from=1-6, to=2-7]
	\arrow[no head, from=2-1, to=1-2]
	\arrow[no head, from=2-1, to=2-2]
	\arrow[no head, from=2-2, to=2-3]
	\arrow[no head, from=2-2, to=3-1]
	\arrow[curve={height=-12pt}, no head, from=2-2, to=4-4]
	\arrow[no head, from=2-3, to=1-5]
	\arrow[no head, from=2-3, to=3-5]
	\arrow[no head, from=2-4, to=1-2]
	\arrow[no head, from=2-4, to=1-6]
	\arrow[no head, from=2-4, to=3-6]
	\arrow[no head, from=2-5, to=1-3]
	\arrow[no head, from=2-5, to=2-6]
	\arrow[no head, from=2-5, to=3-3]
	\arrow[no head, from=2-6, to=2-7]
	\arrow[no head, from=2-6, to=3-7]
	\arrow[no head, from=2-7, to=3-6]
	\arrow[no head, from=3-2, to=2-1]
	\arrow[no head, from=3-2, to=2-4]
	\arrow[no head, from=3-3, to=3-2]
	\arrow[equals, from=3-4, to=2-3]
	\arrow[equals, from=3-4, to=2-4]
	\arrow[equals, from=3-4, to=2-5]
	\arrow[curve={height=-12pt}, no head, from=3-5, to=3-3]
	\arrow[no head, from=3-6, to=3-5]
	\arrow[curve={height=-12pt}, no head, from=4-4, to=2-6]
\end{tikzcd}\]
\caption{$D=28$ (non-principal)}
\end{subfigure}
\hfill
\begin{subfigure}{.45\textwidth}
\centering
\[\begin{tikzcd}[cramped,sep=tiny]
	& {C_2} & {E_5} & {E_3^+} & {C_3} \\
	{\boxed{F_9}} & {\boxed{C_1}} & {F_5} & {F_6} & {F_7} \\
	& {C_5} & {E_3} & {E_2^+} & {C_4}
	\arrow[curve={height=-12pt}, no head, from=1-2, to=1-5]
	\arrow[no head, from=1-4, to=2-5]
	\arrow[curve={height=-12pt}, no head, from=1-4, to=3-4]
	\arrow[no head, from=2-1, to=1-3]
	\arrow[equals, from=2-1, to=2-2]
	\arrow[curve={height=-18pt}, equals, from=2-1, to=2-5]
	\arrow[no head, from=2-1, to=3-3]
	\arrow[no head, from=2-2, to=1-2]
	\arrow[no head, from=2-2, to=1-4]
	\arrow[no head, from=2-2, to=3-2]
	\arrow[no head, from=2-2, to=3-4]
	\arrow[no head, from=2-3, to=1-2]
	\arrow[no head, from=2-3, to=1-3]
	\arrow[no head, from=2-3, to=3-2]
	\arrow[no head, from=2-3, to=3-3]
	\arrow[no head, from=2-4, to=1-3]
	\arrow[no head, from=2-4, to=1-4]
	\arrow[no head, from=2-4, to=3-3]
	\arrow[no head, from=2-4, to=3-4]
	\arrow[no head, from=2-5, to=1-5]
	\arrow[curve={height=12pt}, no head, from=3-2, to=3-5]
	\arrow[no head, from=3-4, to=2-5]
	\arrow[curve={height=18pt}, no head, from=3-5, to=1-5]
	\arrow[no head, from=3-5, to=2-5]
\end{tikzcd}\]
\caption{$D=29$ (principal)}
\end{subfigure}
\end{figure}
\begin{figure}[H]
\begin{subfigure}{.45\textwidth}
\centering
\[\begin{tikzcd}[cramped,sep=tiny]
	{C_{11}} & {C_{12}} & {C_1} && {C_2} & {C_3} & {C_4} \\
	& {F_8^2} && {F_6} && {F_8^1} \\
	{C_{10}} & {C_9} & {C_8} && {C_7} & {C_6} & {C_5}
	\arrow[no head, from=1-1, to=1-2]
	\arrow[no head, from=1-2, to=1-3]
	\arrow[no head, from=1-3, to=1-5]
	\arrow[equals, from=1-3, to=3-3]
	\arrow[no head, from=1-5, to=1-6]
	\arrow[shift left, equals, from=1-5, to=3-5]
	\arrow[no head, from=1-6, to=1-7]
	\arrow[no head, from=1-7, to=3-7]
	\arrow[no head, from=2-2, to=1-1]
	\arrow[no head, from=2-2, to=1-5]
	\arrow[no head, from=2-2, to=3-1]
	\arrow[no head, from=2-2, to=3-5]
	\arrow[no head, from=2-4, to=1-2]
	\arrow[no head, from=2-4, to=1-6]
	\arrow[no head, from=2-4, to=3-2]
	\arrow[no head, from=2-4, to=3-6]
	\arrow[no head, from=2-6, to=1-3]
	\arrow[no head, from=2-6, to=1-7]
	\arrow[no head, from=2-6, to=3-3]
	\arrow[no head, from=2-6, to=3-7]
	\arrow[no head, from=3-1, to=1-1]
	\arrow[no head, from=3-2, to=3-1]
	\arrow[no head, from=3-3, to=3-2]
	\arrow[no head, from=3-5, to=3-3]
	\arrow[no head, from=3-6, to=3-5]
	\arrow[no head, from=3-7, to=3-6]
\end{tikzcd}\]
\caption{$D=33$ (non-principal)}
\end{subfigure}
\hfill
\begin{subfigure}{.45\textwidth}
\centering
\[\begin{tikzcd}[cramped,sep=tiny]
	{C_1} && {E_2^+} & {C_2} & {C_3} \\
	{\boxed{C_7}} & {F_9} & {F_{10}} & {E_3^+} & {F_7} \\
	{C_6} && {E_4^+} & {C_5} & {C_4}
	\arrow[curve={height=-12pt}, no head, from=1-1, to=1-4]
	\arrow[curve={height=-18pt}, no head, from=1-1, to=3-1]
	\arrow[no head, from=1-3, to=2-5]
	\arrow[curve={height=-12pt}, no head, from=1-3, to=3-3]
	\arrow[no head, from=1-4, to=1-5]
	\arrow[curve={height=-12pt}, no head, from=1-5, to=3-5]
	\arrow[no head, from=2-1, to=1-1]
	\arrow[equals, from=2-1, to=2-2]
	\arrow[no head, from=2-2, to=1-5]
	\arrow[equals, from=2-2, to=2-3]
	\arrow[no head, from=2-2, to=3-5]
	\arrow[no head, from=2-3, to=1-3]
	\arrow[equals, from=2-3, to=2-4]
	\arrow[no head, from=2-3, to=3-3]
	\arrow[no head, from=2-4, to=1-1]
	\arrow[no head, from=2-4, to=3-1]
	\arrow[no head, from=2-5, to=1-4]
	\arrow[no head, from=2-5, to=3-4]
	\arrow[no head, from=3-1, to=2-1]
	\arrow[no head, from=3-3, to=2-5]
	\arrow[curve={height=-12pt}, no head, from=3-4, to=3-1]
	\arrow[no head, from=3-5, to=3-4]
\end{tikzcd}\]
\caption{$D=37$ (principal)}
\end{subfigure}
\end{figure}
\begin{figure}[H]
\begin{subfigure}{.45\textwidth}
\centering
\[\begin{tikzcd}[cramped,sep=tiny]
	{C_3} & {C_2} && {F_6} && {C_2'} & {C_3'} \\
	{C_4} & {F_9^2} & {\boxed{C_1}} && {\boxed{C_1'}} & {F_9^1} & {C_4'} \\
	{C_5} & {C_6} & {E_2} & {F_{10}} & {E_4} & {C_6'} & {  C_5'}
	\arrow[no head, from=1-1, to=1-2]
	\arrow[no head, from=1-1, to=2-2]
	\arrow[no head, from=1-2, to=2-3]
	\arrow[no head, from=1-4, to=1-2]
	\arrow[no head, from=1-4, to=1-6]
	\arrow[curve={height=18pt}, no head, from=1-4, to=3-2]
	\arrow[curve={height=-18pt}, no head, from=1-4, to=3-6]
	\arrow[no head, from=1-6, to=1-7]
	\arrow[no head, from=1-7, to=2-7]
	\arrow[no head, from=2-1, to=1-1]
	\arrow[curve={height=-18pt}, equals, from=2-2, to=2-5]
	\arrow[equals, from=2-3, to=2-5]
	\arrow[curve={height=-18pt}, equals, from=2-3, to=2-6]
	\arrow[no head, from=2-3, to=3-2]
	\arrow[no head, from=2-3, to=3-4]
	\arrow[no head, from=2-5, to=1-6]
	\arrow[no head, from=2-5, to=3-4]
	\arrow[no head, from=2-6, to=1-7]
	\arrow[no head, from=2-6, to=3-7]
	\arrow[no head, from=2-7, to=3-7]
	\arrow[no head, from=3-1, to=2-1]
	\arrow[no head, from=3-1, to=2-2]
	\arrow[no head, from=3-2, to=3-1]
	\arrow[no head, from=3-3, to=3-4]
	\arrow[no head, from=3-4, to=2-1]
	\arrow[no head, from=3-4, to=2-7]
	\arrow[no head, from=3-4, to=3-5]
	\arrow[no head, from=3-6, to=2-5]
	\arrow[no head, from=3-7, to=3-6]
\end{tikzcd}\]
\caption{$D=40$ (principal)}
\end{subfigure}
\hfill
\begin{subfigure}{.45\textwidth}
\centering
\[\begin{tikzcd}[cramped,sep=tiny]
	& {C_1} & {C_4} & {C_3} \\
	& {E_2} & {C_2} & {E_1} \\
	{E_2^+} & {F_{10}} & {F_8} & {F_5} & {E_1^+} \\
	& {E_4} & {C_2'} & {E_3} \\
	& {C_3'} & {C_4'} & {C_1'}
	\arrow[no head, from=1-2, to=1-3]
	\arrow[no head, from=1-3, to=1-4]
	\arrow[no head, from=1-4, to=2-3]
	\arrow[no head, from=2-2, to=3-2]
	\arrow[no head, from=2-3, to=1-2]
	\arrow[shift left, curve={height=-6pt}, no head, from=2-3, to=4-3]
	\arrow[shift right, curve={height=6pt}, no head, from=2-3, to=4-3]
	\arrow[no head, from=2-4, to=3-4]
	\arrow[no head, from=3-1, to=1-2]
	\arrow[no head, from=3-1, to=5-2]
	\arrow[no head, from=3-2, to=3-1]
	\arrow[no head, from=3-2, to=3-3]
	\arrow[curve={height=18pt}, no head, from=3-2, to=3-5]
	\arrow[no head, from=3-2, to=4-2]
	\arrow[no head, from=3-3, to=2-3]
	\arrow[no head, from=3-3, to=2-4]
	\arrow[no head, from=3-3, to=4-3]
	\arrow[no head, from=3-3, to=4-4]
	\arrow[no head, from=3-4, to=1-3]
	\arrow[no head, from=3-4, to=5-3]
	\arrow[no head, from=3-5, to=1-4]
	\arrow[no head, from=4-3, to=5-2]
	\arrow[no head, from=4-4, to=3-4]
	\arrow[no head, from=5-2, to=5-3]\hfill
	\arrow[no head, from=5-3, to=5-4]
	\arrow[no head, from=5-4, to=3-5]
	\arrow[no head, from=5-4, to=4-3]
\end{tikzcd}\]
\caption{$D=40$ (non-principal)}
\end{subfigure}
\end{figure}

\begin{figure}[H]
\centering
\begin{subfigure}{.45\textwidth}
\[\begin{tikzcd}[cramped, sep=tiny]
	&& {\boxed{C_{11}}} \\
	{C_{10}} && {F_{10}} && {C_1} \\
	{C_9} && {F_9} && {C_2} \\
	{C_8} & {E_3} & {F_8} & {E_7} & {C_3} \\
	{C_7} & {C_6} & {F_5} & {C_5} & {C_4}
	\arrow[equals, from=1-3, to=2-3]
	\arrow[no head, from=1-3, to=2-5]
	\arrow[shift right, curve={height=6pt}, no head, from=1-3, to=3-3]
	\arrow[shift left, curve={height=-6pt}, no head, from=1-3, to=3-3]
	\arrow[no head, from=2-1, to=1-3]
	\arrow[no head, from=2-3, to=2-1]
	\arrow[no head, from=2-3, to=2-5]
	\arrow[no head, from=2-3, to=4-2]
	\arrow[no head, from=2-3, to=4-4]
	\arrow[no head, from=2-3, to=5-2]
	\arrow[no head, from=2-3, to=5-4]
	\arrow[no head, from=2-5, to=3-5]
	\arrow[no head, from=3-1, to=2-1]
	\arrow[no head, from=3-3, to=4-2]
	\arrow[no head, from=3-3, to=4-4]
	\arrow[no head, from=3-5, to=4-5]
	\arrow[no head, from=4-1, to=3-1]
	\arrow[no head, from=4-2, to=5-3]
	\arrow[no head, from=4-3, to=2-1]
	\arrow[no head, from=4-3, to=2-5]
	\arrow[equals, from=4-3, to=3-3]
	\arrow[no head, from=4-3, to=5-1]
	\arrow[no head, from=4-3, to=5-5]
	\arrow[no head, from=4-5, to=5-5]
	\arrow[no head, from=5-1, to=4-1]
	\arrow[no head, from=5-2, to=5-1]
	\arrow[curve={height=-12pt}, no head, from=5-3, to=3-1]
	\arrow[curve={height=12pt}, no head, from=5-3, to=3-5]
	\arrow[no head, from=5-3, to=4-4]
	\arrow[curve={height=-12pt}, no head, from=5-4, to=5-2]
	\arrow[no head, from=5-5, to=5-4]
\end{tikzcd}\]
\caption{$D=41$ (principal)}
\end{subfigure}
\hfill
\begin{subfigure}{.45\textwidth}
\[\begin{tikzcd}[cramped, sep=tiny]
	& {E_3} & {C_2} & {C_3} & {E_6} \\
	\bm{F_5^1} & {F_9^1} & {\boxed{C_1}} & {\boxed{C_4}} & {F_9^2} & \bm{F_5^2} \\
	& {E_5} & {C_6} & {C_5} & {E_4}
	\arrow[no head, from=1-2, to=2-2]
	\arrow[no head, from=1-3, to=1-4]
	\arrow[no head, from=1-4, to=2-4]
	\arrow[no head, from=1-4, to=2-6]
	\arrow[no head, from=1-5, to=2-6]
	\arrow[no head, from=2-1, to=1-2]
	\arrow[no head, from=2-1, to=1-3]
	\arrow[no head, from=2-1, to=3-2]
	\arrow[no head, from=2-1, to=3-3]
	\arrow[no head, from=2-2, to=3-2]
	\arrow[no head, from=2-3, to=1-3]
	\arrow[equals, from=2-3, to=2-4]
	\arrow[shift right, curve={height=6pt}, no head, from=2-4, to=2-2]
	\arrow[shift left, curve={height=-6pt}, no head, from=2-4, to=2-2]
	\arrow[no head, from=2-4, to=3-4]
	\arrow[no head, from=2-5, to=1-5]
	\arrow[curve={height=18pt}, no head, from=2-5, to=2-2]
	\arrow[curve={height=-18pt}, no head, from=2-5, to=2-2]
	\arrow[shift right, curve={height=6pt}, no head, from=2-5, to=2-3]
	\arrow[shift left, curve={height=-6pt}, no head, from=2-5, to=2-3]
	\arrow[no head, from=2-5, to=3-5]
	\arrow[no head, from=3-3, to=2-3]
	\arrow[no head, from=3-4, to=2-6]
	\arrow[no head, from=3-4, to=3-3]
	\arrow[no head, from=3-5, to=2-6]
\end{tikzcd}\]
\caption{$D=44$ (principal)}
\end{subfigure}
\end{figure}
\begin{figure}[H]
\begin{subfigure}{.45\textwidth}
\[\begin{tikzcd}[cramped,sep=tiny]
	{F_9^2} & \bm{E_5} & {\boxed{C_2}} \\
	& {C_1} & {F_8} & {C_3} \\
	{F_9^1} & \bm{E_6} & {\boxed{C_4}}
	\arrow[no head, from=1-1, to=1-2]
	\arrow[curve={height=-12pt}, no head, from=1-1, to=1-3]
	\arrow[curve={height=12pt}, no head, from=1-1, to=1-3]
	\arrow[no head, from=1-1, to=3-2]
	\arrow[no head, from=1-3, to=2-4]
	\arrow[curve={height=12pt}, no head, from=1-3, to=3-3]
	\arrow[curve={height=-12pt}, no head, from=1-3, to=3-3]
	\arrow[no head, from=2-2, to=1-3]
	\arrow[curve={height=12pt}, no head, from=2-2, to=2-4]
	\arrow[curve={height=-12pt}, no head, from=2-2, to=2-4]
	\arrow[no head, from=2-3, to=1-2]
	\arrow[no head, from=2-3, to=2-2]
	\arrow[no head, from=2-3, to=2-4]
	\arrow[no head, from=2-3, to=3-2]
	\arrow[no head, from=2-4, to=3-3]
	\arrow[no head, from=3-1, to=1-2]
	\arrow[no head, from=3-1, to=3-2]
	\arrow[curve={height=12pt}, no head, from=3-1, to=3-3]
	\arrow[curve={height=-12pt}, no head, from=3-1, to=3-3]
	\arrow[no head, from=3-3, to=2-2]
\end{tikzcd}\]
\caption{$D=56$ (principal)}
\end{subfigure}
\hfill
\begin{subfigure}{.45\textwidth}
\[\begin{tikzcd}[cramped,sep=tiny]
	{C_3} & {C_4} & {\boxed{C_5}} & {C_6} & {C_7} \\
	{C_2} && {F_7^2} && {C_8} \\
	& {E_3^+} & {F_6} & {E_4^+} \\
	{C_1} && {F_7^1} && {C_9} \\
	{C_{14}} & {C_{13}} & {\boxed{C_{12}}} & {C_{11}} & {C_{10}}
	\arrow[no head, from=1-1, to=1-2]
	\arrow[no head, from=1-2, to=1-3]
	\arrow[no head, from=1-3, to=1-4]
	\arrow[curve={height=-12pt}, no head, from=1-3, to=5-3]
	\arrow[curve={height=12pt}, no head, from=1-3, to=5-3]
	\arrow[no head, from=1-4, to=1-5]
	\arrow[no head, from=1-5, to=2-5]
	\arrow[no head, from=2-1, to=1-1]
	\arrow[curve={height=-12pt}, no head, from=2-1, to=2-5]
	\arrow[no head, from=2-1, to=5-3]
	\arrow[no head, from=2-3, to=1-1]
	\arrow[no head, from=2-3, to=1-5]
	\arrow[no head, from=2-5, to=4-5]
	\arrow[no head, from=2-5, to=5-3]
	\arrow[no head, from=3-2, to=2-3]
	\arrow[no head, from=3-2, to=4-3]
	\arrow[no head, from=3-3, to=1-2]
	\arrow[no head, from=3-3, to=1-4]
	\arrow[no head, from=3-3, to=5-2]
	\arrow[no head, from=3-3, to=5-4]
	\arrow[no head, from=3-4, to=2-3]
	\arrow[no head, from=4-1, to=1-3]
	\arrow[no head, from=4-1, to=2-1]
	\arrow[curve={height=12pt}, no head, from=4-1, to=4-5]
	\arrow[no head, from=4-3, to=3-4]
	\arrow[no head, from=4-3, to=5-1]
	\arrow[no head, from=4-3, to=5-5]
	\arrow[no head, from=4-5, to=1-3]
	\arrow[no head, from=4-5, to=5-5]
	\arrow[no head, from=5-1, to=4-1]
	\arrow[no head, from=5-2, to=5-1]
	\arrow[no head, from=5-3, to=5-2]
	\arrow[no head, from=5-4, to=5-3]
	\arrow[no head, from=5-5, to=5-4]
\end{tikzcd}\]
\caption{$D=57$ (principal)}
\end{subfigure}
\end{figure}
\begin{figure}[H]
\begin{subfigure}{.45\textwidth}
\[\begin{tikzcd}[cramped,sep=tiny]
	{\boxed{C_2}} & {E_4^+} && {E_6^+} & {\boxed{C_4}} \\
	& {E_3^+} && {E_5^+} \\
	& {E_8^+} && {E_9^+} \\
	& {E_5} & {F_6} & {E_6} \\
	{C_1} && \bm{E_7^+} && {C_3}
	\arrow[no head, from=1-1, to=1-2]
	\arrow[curve={height=-12pt}, equals, from=1-1, to=1-5]
	\arrow[no head, from=1-1, to=2-4]
	\arrow[no head, from=1-1, to=3-2]
	\arrow[no head, from=1-1, to=5-1]
	\arrow[no head, from=1-4, to=1-5]
	\arrow[no head, from=1-4, to=2-2]
	\arrow[no head, from=1-5, to=5-5]
	\arrow[no head, from=2-2, to=1-5]
	\arrow[no head, from=2-4, to=1-2]
	\arrow[no head, from=3-2, to=1-5]
	\arrow[no head, from=3-4, to=1-1]
	\arrow[no head, from=3-4, to=1-5]
	\arrow[no head, from=3-4, to=4-3]
	\arrow[no head, from=4-3, to=3-2]
	\arrow[no head, from=4-3, to=4-2]
	\arrow[no head, from=4-3, to=4-4]
	\arrow[curve={height=12pt}, no head, from=5-1, to=1-5]
	\arrow[curve={height=12pt}, no head, from=5-1, to=5-5]
	\arrow[no head, from=5-3, to=5-1]
	\arrow[no head, from=5-3, to=5-5]
	\arrow[curve={height=-12pt}, no head, from=5-5, to=1-1]
\end{tikzcd}\]
\caption{$D=69$ (principal)}
\end{subfigure}
\hfill
\begin{subfigure}{.45\textwidth}
\[\begin{tikzcd}[cramped,sep=tiny]
	{C_3'} & {C_5} & {\boxed{C_6}} & {\boxed{C_3}} & {C_4} & {C_4'} \\
	{C_1'} & {C_1} & {\boxed{C_2'}} & {\boxed{C_5'}} & {C_2} & {C_6'}
	\arrow[curve={height=-12pt}, no head, from=1-1, to=1-3]
	\arrow[curve={height=-24pt}, no head, from=1-1, to=1-6]
	\arrow[no head, from=1-1, to=2-3]
	\arrow[no head, from=1-2, to=1-3]
	\arrow[no head, from=1-2, to=2-2]
	\arrow[no head, from=1-2, to=2-3]
	\arrow[equals, from=1-3, to=1-4]
	\arrow[equals, from=1-4, to=2-4]
	\arrow[curve={height=12pt}, no head, from=1-5, to=1-2]
	\arrow[no head, from=1-5, to=1-4]
	\arrow[no head, from=1-5, to=2-4]
	\arrow[curve={height=12pt}, no head, from=1-6, to=1-4]
	\arrow[no head, from=1-6, to=2-6]
	\arrow[no head, from=2-1, to=1-1]
	\arrow[no head, from=2-1, to=1-3]
	\arrow[curve={height=12pt}, no head, from=2-1, to=2-3]
	\arrow[curve={height=24pt}, no head, from=2-1, to=2-6]
	\arrow[no head, from=2-2, to=1-3]
	\arrow[no head, from=2-2, to=2-3]
	\arrow[curve={height=12pt}, no head, from=2-2, to=2-5]
	\arrow[equals, from=2-3, to=1-3]
	\arrow[no head, from=2-4, to=1-6]
	\arrow[equals, from=2-4, to=2-3]
	\arrow[no head, from=2-5, to=1-4]
	\arrow[no head, from=2-5, to=1-5]
	\arrow[no head, from=2-5, to=2-4]
	\arrow[no head, from=2-6, to=1-4]
	\arrow[curve={height=-12pt}, no head, from=2-6, to=2-4]
\end{tikzcd}\]
\caption{$D=105$ (principal)}
\end{subfigure}
\end{figure}

\pagebreak
\section{Divisors on $\tilde X_g$}\label{Sec:computations}
 The following tables were computed with the aid of Magma \cite{Magma}. Code to this effect is available at \href{https://github.com/dgvirtz/hilbmodcycles}{github.com/dgvirtz/hilbmodcycles}.

\paragraph{Cusp divisors.\ }
We list the ideals defining the cusps, the negative of the self-intersections of each cycle of cusp divisors and the quadratic forms $Q_k(p,q)=\Nm(\lambda_{p,q})=ap^2+bpq+cq^2=:[a,b,c]$.
The raw count of cusp divisors may be compared with \cite[Table 2, row $l$]{VDG}.

\small
\begin{longtable}{ |Sr|Sl|Sc|}
\caption{Cusp divisors. $^*$ = The resolution cycle of the cusp is double the displayed cycle. $^\dagger$ = There are two cusps with the same resolution cycle.}\\
\hline 
\multicolumn{1}{|c|}{$D$} & \multicolumn{1}{c|}{Cycle} & \multicolumn{1}{c|}{$Q_k$'s} \\
\hline\hline
\multicolumn{3}{|c|}{non-principal $g$}\\
\hline
$21$ & $[ -3, -2, -2]^*$& \makecell{$[3,9,5]$, $[5,11,5]$, $[5,9,3]$}\\ 
 \hline
$24$ & $[ -4, -2, -2, -2]^*$ & \makecell{$[2,8,5]$, $[5,12,6]$, $[6,12,5]$, $[5,8,2]$}\\ 
 \hline
$28$ & $[ -3, -3, -2, -2, -2]^*$ & \makecell{$[3,8,3]$, $[3,10,6]$, $[6,14,7]$, $[7,14,6]$, $[6,10,3]$}\\ 
 \hline
$33$ & $[ -4, -4, -2, -2, -2, -2]^*$ & \makecell{$[2,7,2]$, $[2,9,6]$, $[6,15,8]$, $[8,17,8]$, $[8,15,6]$, $[6,9,2]$}\\ 
 \hline
$40$ & $[ -3, -4, -3, -2 ]^\dagger$ & \makecell{$[3,8,2]$, $[2,8,3]$, $[3,10,5]$, $[5,10,3]$}\\ 
 \hline
\multicolumn{3}{|c|}{principal $g$}\\
\hline
$29$ & $[ -7, -2, -2, -2, -2 ]$ & \makecell{$[1,7,5]$, $[5,13,7]$, $[7,15,7]$, $[7,13,5]$, $[5,7,1]$}\\ 
 \hline
$37$ & $[ -3, -2, -2, -2, -2, -3, -7 ]$ & \makecell{$[3,11,7]$, $[7,17,9]$, $[9,19,9]$, $[9,17,7]$,\\ $[7,11,3]$, $[3,7,1]$, $[1,7,3]$}\\ 
 \hline
$40$ & $[ -8, -2, -2, -2, -2, -2 ]^\dagger$ & \makecell{$[1,8,6]$, $[6,16,9]$, $[9,20,10]$, $[10,20,9]$, $[9,16,6]$, $[6,8,1]$}\\ 
 \hline
$41$ & \makecell{$[ -4, -2, -3, -2, -2, -2$,\\ $-2, -3, -2, -4, -7 ]$} & \makecell{$[2,9,5]$, $[5,11,4]$, $[4,13,8]$, $[8,19,10]$, $[10,21,10]$,\\ $[10,19,8]$, $[8,13,4]$, $[4,11,5]$, $[5,9,2]$, $[2,7,1]$, $[1,7,2]$}\\ 
 \hline
$44$ & $[ -8, -2, -2]^*$ & \makecell{$[1,8,5]$, $[5,12,5]$, $[5,8,1]$}\\ 
 \hline
$56$ & $[ -4, -8]^*$ & \makecell{$[2,8,1]$, $[1,8,2]$}\\ 
 \hline
$57$ & $[ -3, -3, -2, -2, -9, -2, -2]^*$ & \makecell{$[4,11,4]$, $[4,13,7]$, $[7,15,6]$, $[6,9,1]$,\\ $[1,9,6]$, $[6,15,7]$, $[7,13,4]$}\\ 
 \hline
$69$ & $[ -3, -9]^*$ & \makecell{$[3,9,1]$, $[1,9,3]$}\\ 
 \hline
$105$ & $[ -3, -3, -11]^{*\dagger}$ & \makecell{$[4,13,4]$, $[4,11,1]$, $[1,11,4]$}\\ 
\hline
 \endlastfoot
\end{longtable}
\normalsize

\paragraph{Elliptic divisors.~}
We list the stabiliser matrices of elliptic points. Below, the norm $A$ of the genus ideal is chosen as small as possible (i.e.\ $A=5,5,3,2,3$ for $D=21,24,28,33,40$ in the non-principal cases and $A=1$ in the principal cases).
The raw count of elliptic points may be be compared with the class number-theoretic formulae in \cite[Table 1.I]{VDG} and \cite[Table 2]{VDG}.

\pagebreak

\small
\begin{longtable}{ |Sr|p{0.9\textwidth}|}
\caption{Elliptic divisors of order $2$. $\beta^2=$ square-free part of $D$.}\\
\hline
\multicolumn{1}{|c|}{$D$} & \multicolumn{1}{c|}{Stabiliser matrices} \\
\hline\hline
\multicolumn{2}{|c|}{non-principal $g$}\\
\hline \rule{0pt}{4mm}
$21$ & { \scalebox{0.8}{$ {\begin{pmatrix} 0 & \frac{-\beta + 1}{10} \\ \frac{\beta + 1}{2} & 0 \end{pmatrix}}$},
\scalebox{0.8}{ $ \begin{pmatrix} 0 & \frac{-\beta - 4}{5} \\ \beta - 4 & 0 \end{pmatrix}$},
\scalebox{0.8}{ $ \begin{pmatrix} 2 & -1 \\ 5 & -2 \end{pmatrix}$},
\scalebox{0.8}{ $ \begin{pmatrix} -2 & -1 \\ 5 & 2 \end{pmatrix}$}
} \vspace{1mm}\\ 
\hline \rule{0pt}{4mm}
$24$ & {\scalebox{0.8}{ $ \begin{pmatrix} 0 & \frac{-\beta + 1}{5} \\ \beta + 1 & 0 \end{pmatrix}$},\nobreak
\scalebox{0.8}{ $ \begin{pmatrix} 0 & \frac{-7\beta + 17}{5} \\ 7\beta + 17 & 0 \end{pmatrix}$},\nobreak
\scalebox{0.8}{ $ \begin{pmatrix} 2 & -1 \\ 5 & -2 \end{pmatrix}$},\nobreak
\scalebox{0.8}{ $ \begin{pmatrix} -2 & -1 \\ 5 & 2 \end{pmatrix}$},\nobreak
\scalebox{0.8}{ $ \begin{pmatrix} \beta - 3 & \frac{4\beta - 14}{5} \\ -\beta + 4 & -\beta + 3 \end{pmatrix}$},\nobreak
\scalebox{0.8}{ $ \begin{pmatrix} \beta - 1 & -\beta + 2 \\ 2\beta + 2 & -\beta + 1 \end{pmatrix}$}
}\vspace{1mm}\\ 
 \hline \rule{0pt}{4mm} 
$28$ & {\scalebox{0.8}{ $ \begin{pmatrix} 0 & \frac{-\beta - 2}{3} \\ \beta - 2 & 0 \end{pmatrix}$},
\scalebox{0.8}{ $ \begin{pmatrix} 0 & \frac{-14\beta - 37}{3} \\ 14\beta - 37 & 0 \end{pmatrix}$},
\scalebox{0.8}{ $ \begin{pmatrix} -\beta & \frac{-4\beta - 8}{3} \\ 2\beta - 4 & \beta \end{pmatrix}$},
\scalebox{0.8}{ $ \begin{pmatrix} -\beta - 2 & \frac{-4\beta - 8}{3} \\ \beta + 1 & \beta + 2 \end{pmatrix}$}
}\vspace{1mm}\\  
 \hline \rule{0pt}{4mm} 
$33$ & {\scalebox{0.8}{ $ \begin{pmatrix} 0 & \frac{-\beta + 5}{4} \\ \frac{\beta + 5}{2} & 0 \end{pmatrix}$},
\scalebox{0.8}{ $ \begin{pmatrix} 0 & \frac{-3\beta - 17}{4} \\ \frac{3\beta - 17}{2} & 0 \end{pmatrix}$},
\scalebox{0.8}{ $ \begin{pmatrix} -\beta + 6 & \frac{7\beta - 47}{4} \\ \frac{-\beta + 7}{2} & \beta - 6 \end{pmatrix}$},
\scalebox{0.8}{ $ \begin{pmatrix} \frac{-231\beta - 1327}{2} & \frac{223\beta - 2947}{4} \\ 368\beta + 2114 & \frac{231\beta + 
1327}{2} \end{pmatrix}$}
}\vspace{1mm}\\  
 \hline \rule{0pt}{4mm} 
$40$ & {\scalebox{0.8}{ $ \begin{pmatrix} -2 & \frac{-\beta - 5}{3} \\ -\beta + 5 & 2 \end{pmatrix}$},
\scalebox{0.8}{ $ \begin{pmatrix} -\beta - 1 & \frac{-2\beta - 4}{3} \\ 2\beta - 1 & \beta + 1 \end{pmatrix}$},
\scalebox{0.8}{ $ \begin{pmatrix} 2 & \frac{-\beta - 5}{3} \\ -\beta + 5 & -2 \end{pmatrix}$},
\scalebox{0.8}{ $ \begin{pmatrix} \beta - 3 & \frac{-2\beta + 5}{3} \\ 2\beta - 4 & -\beta + 3 \end{pmatrix}$},
\scalebox{0.8}{ $ \begin{pmatrix} \beta + 3 & \frac{-2\beta - 10}{3} \\ \beta + 4 & -\beta - 3 \end{pmatrix}$},
\scalebox{0.8}{ $ \begin{pmatrix} -1 & \frac{-\beta - 2}{3} \\ \beta - 2 & 1 \end{pmatrix}$}
}\vspace{1mm}\\  
\hline
\multicolumn{2}{|c|}{principal $g$}\\
\hline \rule{0pt}{4mm}
$29$ & {\scalebox{0.8}{ $ \begin{pmatrix} 0 & -1 \\ 1 & 0 \end{pmatrix}$},
\scalebox{0.8}{ $ \begin{pmatrix} 0 & \frac{-\beta - 5}{2} \\ \frac{\beta - 5}{2} & 0 \end{pmatrix}$},
\scalebox{0.8}{ $ \begin{pmatrix} -\beta - 5 & \frac{-3\beta - 19}{2} \\ \frac{\beta + 7}{2} & \beta + 5 \end{pmatrix}$},
\scalebox{0.8}{ $ \begin{pmatrix} 2 & \frac{-\beta - 3}{2} \\ \frac{\beta - 3}{2} & -2 \end{pmatrix}$},
\scalebox{0.8}{ $ \begin{pmatrix} \frac{-\beta - 3}{2} & \frac{-\beta - 7}{2} \\ 3 & \frac{\beta + 3}{2} \end{pmatrix}$},
\scalebox{0.8}{ $ \begin{pmatrix} \frac{\beta + 7}{2} & -\beta - 4 \\ \frac{\beta + 3}{2} & \frac{-\beta - 7}{2} \end{pmatrix}$}
}\vspace{1mm}\\  
 \hline \rule{0pt}{4mm}
$37$ & {\scalebox{0.8}{ $ \begin{pmatrix} 0 & -1 \\ 1 & 0 \end{pmatrix}$},
\scalebox{0.8}{ $ \begin{pmatrix} 0 & -\beta - 6 \\ \beta - 6 & 0 \end{pmatrix}$}
}\vspace{1mm}\\  
 \hline \rule{0pt}{4mm}
$40$ & {\scalebox{0.8}{ $ \begin{pmatrix} 0 & -1 \\ 1 & 0 \end{pmatrix}$},
\scalebox{0.8}{ $ \begin{pmatrix} 0 & -\beta + 3 \\ \beta + 3 & 0 \end{pmatrix}$},
\scalebox{0.8}{ $ \begin{pmatrix} 3\beta & -13 \\ 7 & -3\beta \end{pmatrix}$},
\scalebox{0.8}{ $ \begin{pmatrix} 3\beta + 30 & -13\beta - 9 \\ 7\beta + 9 & -3\beta - 30 \end{pmatrix}$},
\scalebox{0.8}{ $ \begin{pmatrix} \beta + 1 & -\beta - 6 \\ 2 & -\beta - 1 \end{pmatrix}$},
\scalebox{0.8}{ $ \begin{pmatrix} 3 & -\beta \\ \beta & -3 \end{pmatrix}$}
}\vspace{1mm}\\  
 \hline \rule{0pt}{4mm}
$41$ & {\scalebox{0.8}{ $ \begin{pmatrix} 0 & -1 \\ 1 & 0 \end{pmatrix}$},
\scalebox{0.8}{ $ \begin{pmatrix} 0 & -5\beta - 32 \\ 5\beta - 32 & 0 \end{pmatrix}$},
\scalebox{0.8}{ $ \begin{pmatrix} \frac{-\beta - 3}{2} & \frac{-\beta - 9}{2} \\ 3 & \frac{\beta + 3}{2} \end{pmatrix}$},
\scalebox{0.8}{ $ \begin{pmatrix} -3 & \frac{-\beta - 1}{2} \\ \frac{\beta - 1}{2} & 3 \end{pmatrix}$},
\scalebox{0.8}{ $ \begin{pmatrix} -\beta - 6 & \frac{-3\beta - 27}{2} \\ \frac{\beta + 7}{2} & \beta + 6 \end{pmatrix}$},
\scalebox{0.8}{ $ \begin{pmatrix} 2\beta - 13 & \frac{-7\beta + 25}{2} \\ \frac{3\beta - 19}{2} & -2\beta + 13 \end{pmatrix}$},
\scalebox{0.8}{ $ \begin{pmatrix} \frac{\beta + 3}{2} & \frac{-\beta - 9}{2} \\ 3 & \frac{-\beta - 3}{2} \end{pmatrix}$},
\scalebox{0.8}{ $ \begin{pmatrix} 3 & \frac{-\beta - 1}{2} \\ \frac{\beta - 1}{2} & -3 \end{pmatrix}$}
}\vspace{1mm}\\  
 \hline \rule{0pt}{4mm}
$44$ & {\scalebox{0.8}{ $ \begin{pmatrix} 0 & -1 \\ 1 & 0 \end{pmatrix}$},
\scalebox{0.8}{ $ \begin{pmatrix} 0 & -3\beta - 10 \\ -3\beta + 10 & 0 \end{pmatrix}$},
\scalebox{0.8}{ $ \begin{pmatrix} -\beta & -4 \\ 3 & \beta \end{pmatrix}$},
\scalebox{0.8}{ $ \begin{pmatrix} 2\beta - 6 & 3\beta - 12 \\ -\beta + 4 & -2\beta + 6 \end{pmatrix}$},
\scalebox{0.8}{ $ \begin{pmatrix} \beta & -4 \\ 3 & -\beta \end{pmatrix}$},
\scalebox{0.8}{ $ \begin{pmatrix} -2\beta + 6 & 3\beta - 12 \\ -\beta + 4 & 2\beta - 6 \end{pmatrix}$},
\scalebox{0.8}{ $ \begin{pmatrix} -\beta & -6 \\ 2 & \beta \end{pmatrix}$},
\scalebox{0.8}{ $ \begin{pmatrix} -\beta - 4 & -2\beta - 2 \\ \beta + 3 & \beta + 4 \end{pmatrix}$},
\scalebox{0.8}{ $ \begin{pmatrix} -3 & -\beta + 1 \\ \beta + 1 & 3 \end{pmatrix}$},
\scalebox{0.8}{ $ \begin{pmatrix} 3 & -\beta + 1 \\ \beta + 1 & -3 \end{pmatrix}$}
}\vspace{0.1mm}\\  
 \hline \rule{0pt}{4mm}
$56$ & {\scalebox{0.8}{ $ \begin{pmatrix} 0 & -1 \\ 1 & 0 \end{pmatrix}$},
\scalebox{0.8}{ $ \begin{pmatrix} 0 & 4\beta - 15 \\ 4\beta + 15 & 0 \end{pmatrix}$},
\scalebox{0.8}{ $ \begin{pmatrix} -\beta + 5 & -2\beta + 4 \\ \beta - 3 & \beta - 5 \end{pmatrix}$},
\scalebox{0.8}{ $ \begin{pmatrix} 5 & -2\beta + 2 \\ \beta + 1 & -5 \end{pmatrix}$},
\scalebox{0.8}{ $ \begin{pmatrix} -\beta & -5 \\ 3 & \beta \end{pmatrix}$},
\scalebox{0.8}{ $ \begin{pmatrix} \beta & -5 \\ 3 & -\beta \end{pmatrix}$},
\scalebox{0.8}{ $ \begin{pmatrix} \beta - 5 & -2\beta + 4 \\ \beta - 3 & -\beta + 5 \end{pmatrix}$},
\scalebox{0.8}{ $ \begin{pmatrix} -5 & -2\beta + 2 \\ \beta + 1 & 5 \end{pmatrix}$},
\scalebox{0.8}{ $ \begin{pmatrix} -2\beta + 7 & 3\beta - 16 \\ -\beta + 4 & 2\beta - 7 \end{pmatrix}$},
\scalebox{0.8}{ $ \begin{pmatrix} -\beta + 5 & -2\beta + 6 \\ \beta - 2 & \beta - 5 \end{pmatrix}$},
\scalebox{0.8}{ $ \begin{pmatrix} \beta + 1 & -\beta - 8 \\ 2 & -\beta - 1 \end{pmatrix}$},
\scalebox{0.8}{ $ \begin{pmatrix} \beta - 5 & -2\beta + 6 \\ \beta - 2 & -\beta + 5 \end{pmatrix}$}
}\vspace{1mm}\\  
 \hline \rule{0pt}{4mm}
$57$ & {\scalebox{0.8}{ $ \begin{pmatrix} 0 & -1 \\ 1 & 0 \end{pmatrix}$},
\scalebox{0.8}{ $ \begin{pmatrix} 0 & 20\beta - 151 \\ 20\beta + 151 & 0 \end{pmatrix}$},
\scalebox{0.8}{ $ \begin{pmatrix} \frac{-\beta + 9}{2} & -\beta + 2 \\ \frac{\beta - 7}{2} & \frac{\beta - 9}{2} \end{pmatrix}$},
\scalebox{0.8}{ $ \begin{pmatrix} -9\beta - 68 & \frac{-35\beta - 181}{2} \\ \frac{11\beta + 83}{2} & 9\beta + 68 \end{pmatrix}$}
}\vspace{1mm}\\  
 \hline \rule{0pt}{4mm}
$69$ & {\scalebox{0.8}{ $ \begin{pmatrix} 0 & -1 \\ 1 & 0 \end{pmatrix}$},
\scalebox{0.8}{ $ \begin{pmatrix} 0 & \frac{-3\beta - 25}{2} \\ \frac{-3\beta + 25}{2} & 0 \end{pmatrix}$},
\scalebox{0.8}{ $ \begin{pmatrix} \frac{\beta - 11}{2} & -\beta + 4 \\ \frac{\beta - 7}{2} & \frac{-\beta + 11}{2} \end{pmatrix}$},
\scalebox{0.8}{ $ \begin{pmatrix} -4 & \frac{-\beta + 1}{2} \\ \frac{\beta + 1}{2} & 4 \end{pmatrix}$},
\scalebox{0.8}{ $ \begin{pmatrix} -\beta - 6 & -\beta - 13 \\ \frac{\beta + 11}{2} & \beta + 6 \end{pmatrix}$},
\scalebox{0.8}{ $ \begin{pmatrix} \beta - 6 & \beta - 13 \\ \frac{-\beta + 11}{2} & -\beta + 6 \end{pmatrix}$},
\scalebox{0.8}{ $ \begin{pmatrix} -\beta - 9 & \frac{-5\beta - 37}{2} \\ \frac{\beta + 7}{2} & \beta + 9 \end{pmatrix}$},
\scalebox{0.8}{ $ \begin{pmatrix} 4 & \frac{-\beta + 1}{2} \\ \frac{\beta + 1}{2} & -4 \end{pmatrix}$}
}\vspace{1mm}\\  
 \hline \rule{0pt}{4mm}
$105$ & {\scalebox{0.8}{ $ \begin{pmatrix} 0 & -1 \\ 1 & 0 \end{pmatrix}$},
\scalebox{0.8}{ $ \begin{pmatrix} 0 & -4\beta - 41 \\ -4\beta + 41 & 0 \end{pmatrix}$},
\scalebox{0.8}{ $ \begin{pmatrix} -5 & \frac{-\beta - 1}{2} \\ \frac{\beta - 1}{2} & 5 \end{pmatrix}$},
\scalebox{0.8}{ $ \begin{pmatrix} -5 & \frac{-\beta + 1}{2} \\ \frac{\beta + 1}{2} & 5 \end{pmatrix}$},
\scalebox{0.8}{ $ \begin{pmatrix} \frac{21\beta - 145}{2} & 14\beta - 241 \\ \frac{-7\beta + 97}{2} & \frac{-21\beta + 145}{2} \end{pmatrix}$},
\scalebox{0.8}{ $ \begin{pmatrix} \frac{49\beta + 355}{2} & -46\beta - 641 \\ \frac{13\beta + 197}{2} & \frac{-49\beta - 355}{2} \end{pmatrix}$},
\scalebox{0.8}{ $ \begin{pmatrix} \frac{-855\beta + 9007}{2} & -969\beta + 9584 \\ \frac{399\beta - 4001}{2} & \frac{855\beta - 
9007}{2} \end{pmatrix}$},
\scalebox{0.8}{ $ \begin{pmatrix} \frac{5\beta + 275}{2} & -33\beta - 92 \\ \frac{11\beta + 11}{2} & \frac{-5\beta - 275}{2} \end{pmatrix}$}
}\vspace{1mm}\\  
 \hline
 \end{longtable}
\normalsize

\pagebreak

\small
\begin{longtable}{ |Sr|p{0.9\textwidth}|}
\caption{Elliptic divisors of type $3^+$.}\\
\hline 
\multicolumn{1}{|c|}{$D$} & \multicolumn{1}{c|}{Stabiliser matrices} \\
\hline\hline
\multicolumn{2}{|c|}{non-principal $g$}\\
\hline \rule{0pt}{4mm}
$21$ & {\scalebox{0.8}{ $ \begin{pmatrix} 1 & \frac{-\beta - 9}{10} \\ \frac{-\beta + 9}{2} & -2 \end{pmatrix}$}
}\vspace{1mm}\\  
 \hline
$24$ & {--}\\ 
 \hline \rule{0pt}{4mm}
$28$ & {\scalebox{0.8}{ $ \begin{pmatrix} -\beta - 3 & \frac{-2\beta - 7}{3} \\ \beta + 4 & \beta + 2 \end{pmatrix}$},
\scalebox{0.8}{ $ \begin{pmatrix} \beta + 3 & \frac{-7\beta - 20}{3} \\ 3 & -\beta - 4 \end{pmatrix}$}
}\vspace{1mm}\\  
 \hline
$33$ & {--}\\ 
 \hline \rule{0pt}{4mm}
$40$ & {\scalebox{0.8}{ $ \begin{pmatrix} -\beta + 3 & \frac{\beta - 7}{3} \\ -2\beta + 7 & \beta - 4 \end{pmatrix}$},
\scalebox{0.8}{ $ \begin{pmatrix} -2 & -1 \\ 3 & 1 \end{pmatrix}$}
}\vspace{1mm}\\  
\hline
\multicolumn{2}{|c|}{principal $g$}\\ 
\hline \rule{0pt}{4mm}
$29$ & {\scalebox{0.8}{ $ \begin{pmatrix} 0 & -1 \\ 1 & -1 \end{pmatrix}$},
\scalebox{0.8}{ $ \begin{pmatrix} \frac{-\beta - 1}{2} & -4 \\ 2 & \frac{\beta - 1}{2} \end{pmatrix}$},
\scalebox{0.8}{ $ \begin{pmatrix} \frac{-\beta - 3}{2} & \frac{-\beta - 9}{2} \\ 2 & \frac{\beta + 1}{2} \end{pmatrix}$}
}\vspace{1mm}\\  
 \hline \rule{0pt}{4mm}
$37$ & {\scalebox{0.8}{ $ \begin{pmatrix} 0 & -1 \\ 1 & -1 \end{pmatrix}$},
\scalebox{0.8}{ $ \begin{pmatrix} \frac{\beta + 1}{2} & \frac{-\beta - 11}{2} \\ 2 & \frac{-\beta - 3}{2} \end{pmatrix}$},
\scalebox{0.8}{ $ \begin{pmatrix} -\beta - 6 & \frac{-3\beta - 23}{2} \\ \frac{\beta + 7}{2} & \beta + 5 \end{pmatrix}$},
\scalebox{0.8}{ $ \begin{pmatrix} \frac{\beta - 1}{2} & -5 \\ 2 & \frac{-\beta - 1}{2} \end{pmatrix}$}
}\vspace{1mm}\\  
 \hline \rule{0pt}{4mm}
$40$ & {\scalebox{0.8}{ $ \begin{pmatrix} 0 & -1 \\ 1 & -1 \end{pmatrix}$},
\scalebox{0.8}{ $ \begin{pmatrix} 3\beta - 5 & 2\beta - 13 \\ -\beta + 7 & -3\beta + 4 \end{pmatrix}$}
}\vspace{1mm}\\  
 \hline \rule{0pt}{4mm}
$41$ & {\scalebox{0.8}{ $ \begin{pmatrix} -1 & -1 \\ 1 & 0 \end{pmatrix}$}
}\vspace{1mm}\\  
 \hline \rule{0pt}{4mm}
$44$ & {\scalebox{0.8}{ $ \begin{pmatrix} 0 & -1 \\ 1 & -1 \end{pmatrix}$},
\scalebox{0.8}{ $ \begin{pmatrix} -5\beta - 17 & -6\beta - 35 \\ 3\beta + 10 & 5\beta + 16 \end{pmatrix}$}
}\vspace{1mm}\\  
 \hline \rule{0pt}{4mm}
$56$ & {\scalebox{0.8}{ $ \begin{pmatrix} 0 & -1 \\ 1 & -1 \end{pmatrix}$},
\scalebox{0.8}{ $ \begin{pmatrix} -2\beta + 7 & -2\beta - 15 \\ -4\beta + 15 & 2\beta - 8 \end{pmatrix}$}
}\vspace{1mm}\\  
 \hline \rule{0pt}{4mm}
$57$ & {\scalebox{0.8}{ $ \begin{pmatrix} 0 & -1 \\ 1 & -1 \end{pmatrix}$},
\scalebox{0.8}{ $ \begin{pmatrix} 10\beta + 75 & 10\beta - 151 \\ 20\beta + 151 & -10\beta - 76 \end{pmatrix}$},
\scalebox{0.8}{ $ \begin{pmatrix} \frac{-\beta - 1}{2} & -5 \\ 3 & \frac{\beta - 1}{2} \end{pmatrix}$},
\scalebox{0.8}{ $ \begin{pmatrix} \frac{\beta - 1}{2} & -5 \\ 3 & \frac{-\beta - 1}{2} \end{pmatrix}$}
}\vspace{1mm}\\  
 \hline \rule{0pt}{4mm}
$69$ & {\scalebox{0.8}{ $ \begin{pmatrix} 0 & -1 \\ 1 & -1 \end{pmatrix}$},
\scalebox{0.8}{ $ \begin{pmatrix} \frac{-7\beta - 59}{2} & -7\beta - 77 \\ \frac{3\beta + 25}{2} & \frac{7\beta + 57}{2} \end{pmatrix}$},
\scalebox{0.8}{ $ \begin{pmatrix} \frac{-\beta - 3}{2} & \frac{-\beta - 19}{2} \\ 2 & \frac{\beta + 1}{2} \end{pmatrix}$},
\scalebox{0.8}{ $ \begin{pmatrix} \frac{-\beta - 5}{2} & \frac{-\beta - 11}{2} \\ 4 & \frac{\beta + 3}{2} \end{pmatrix}$},
\scalebox{0.8}{ $ \begin{pmatrix} \frac{\beta + 3}{2} & \frac{-\beta - 11}{2} \\ 4 & \frac{-\beta - 5}{2} \end{pmatrix}$},
\scalebox{0.8}{ $ \begin{pmatrix} \frac{-\beta - 1}{2} & -9 \\ 2 & \frac{\beta - 1}{2} \end{pmatrix}$},
\scalebox{0.8}{ $ \begin{pmatrix} -\beta + 7 & \frac{3\beta - 33}{2} \\ \frac{-\beta + 9}{2} & \beta - 8 \end{pmatrix}$},
\scalebox{0.8}{ $ \begin{pmatrix} \frac{\beta - 1}{2} & -6 \\ 3 & \frac{-\beta - 1}{2} \end{pmatrix}$},
\scalebox{0.8}{ $ \begin{pmatrix} \frac{-\beta - 1}{2} & -6 \\ 3 & \frac{\beta - 1}{2} \end{pmatrix}$}
}\vspace{1mm}\\  
 \hline \rule{0pt}{4mm}
$105$ & {\scalebox{0.8}{ $ \begin{pmatrix} 0 & -1 \\ 1 & -1 \end{pmatrix}$},
\scalebox{0.8}{ $ \begin{pmatrix} -2\beta + 20 & -2\beta - 41 \\ -4\beta + 41 & 2\beta - 21 \end{pmatrix}$},
\scalebox{0.8}{ $ \begin{pmatrix} 13\beta + 79 & -20\beta - 329 \\ 3\beta + 54 & -13\beta - 80 \end{pmatrix}$},
\scalebox{0.8}{ $ \begin{pmatrix} \frac{21\beta - 121}{2} & 11\beta - 244 \\ -3\beta + 48 & \frac{-21\beta + 119}{2} \end{pmatrix}$},
\scalebox{0.8}{ $ \begin{pmatrix} \frac{\beta - 1}{2} & -9 \\ 3 & \frac{-\beta - 1}{2} \end{pmatrix}$},
\scalebox{0.8}{ $ \begin{pmatrix} \frac{-\beta - 1}{2} & -9 \\ 3 & \frac{\beta - 1}{2} \end{pmatrix}$}
}\vspace{1mm}\\  
 \hline
\end{longtable}
\normalsize

\pagebreak

\small
\begin{longtable}{ |Sr|p{0.9\textwidth}|}
\caption{Elliptic divisors of type $3^-$.}\\
\hline 
\multicolumn{1}{|c|}{$D$} & \multicolumn{1}{c|}{Stabiliser matrices} \\
\endfirsthead
\hline\hline
\multicolumn{2}{|c|}{non-principal $g$}\\
\hline \rule{0pt}{4mm}
$21$ & {\scalebox{0.8}{ $ \begin{pmatrix} 0 & \frac{-\beta + 1}{10} \\ \frac{\beta + 1}{2} & -1 \end{pmatrix}$},
\scalebox{0.8}{ $ \begin{pmatrix} \frac{\beta - 5}{2} & \frac{-2\beta + 2}{5} \\ \beta - 4 & \frac{-\beta + 3}{2} \end{pmatrix}$},
\scalebox{0.8}{ $ \begin{pmatrix} \frac{\beta + 1}{2} & \frac{-3\beta - 7}{10} \\ \beta + 1 & \frac{-\beta - 3}{2} \end{pmatrix}$},
\scalebox{0.8}{ $ \begin{pmatrix} \frac{\beta - 1}{2} & \frac{-3\beta + 3}{10} \\ \beta + 1 & \frac{-\beta - 1}{2} \end{pmatrix}$}
}\vspace{1mm}\\  
 \hline \rule{0pt}{4mm}
$24$ & {\scalebox{0.8}{ $ \begin{pmatrix} 0 & \frac{-\beta + 1}{5} \\ \beta + 1 & -1 \end{pmatrix}$},
\scalebox{0.8}{ $ \begin{pmatrix} -\beta + 2 & \frac{-4\beta - 1}{5} \\ 3\beta - 7 & \beta - 3 \end{pmatrix}$},
\scalebox{0.8}{ $ \begin{pmatrix} -2\beta + 4 & \frac{-12\beta + 27}{5} \\ 2\beta - 3 & 2\beta - 5 \end{pmatrix}$}
}\vspace{1mm}\\  
 \hline \rule{0pt}{4mm}
$28$ & {\scalebox{0.8}{ $ \begin{pmatrix} 0 & \frac{-\beta - 2}{3} \\ \beta - 2 & -1 \end{pmatrix}$},
\scalebox{0.8}{ $ \begin{pmatrix} -4\beta - 11 & \frac{-26\beta - 61}{3} \\ 2\beta + 5 & 4\beta + 10 \end{pmatrix}$}
}\vspace{1mm}\\  
 \hline \rule{0pt}{4mm}
$33$ & {\scalebox{0.8}{ $ \begin{pmatrix} 0 & \frac{-\beta + 5}{4} \\ \frac{\beta + 5}{2} & -1 \end{pmatrix}$},
\scalebox{0.8}{ $ \begin{pmatrix} -2\beta + 11 & \frac{-13\beta + 49}{4} \\ \frac{3\beta - 17}{2} & 2\beta - 12 \end{pmatrix}$},
\scalebox{0.8}{ $ \begin{pmatrix} \frac{-\beta + 5}{2} & \frac{-3\beta + 15}{4} \\ \frac{\beta - 3}{2} & \frac{\beta - 7}{2} \end{pmatrix}$}
}\vspace{1mm}\\  
 \hline \rule{0pt}{4mm}
$40$ & {\scalebox{0.8}{ $ \begin{pmatrix} -2 & \frac{-\beta + 1}{3} \\ \beta + 1 & 1 \end{pmatrix}$},
\scalebox{0.8}{ $ \begin{pmatrix} -25\beta + 109 & \frac{-185\beta + 479}{3} \\ 20\beta - 37 & 25\beta - 110 \end{pmatrix}$}
}\vspace{1mm}\\  
 \hline
 \multicolumn{2}{|c|}{principal $g$}\\
\hline \rule{0pt}{4mm}
$29$ & {\scalebox{0.8}{ $ \begin{pmatrix} \frac{\beta - 7}{2} & -\beta + 1 \\ \frac{\beta - 5}{2} & \frac{-\beta + 5}{2} \end{pmatrix}$},
\scalebox{0.8}{ $ \begin{pmatrix} 2 & \frac{-\beta - 1}{2} \\ \frac{\beta - 1}{2} & -3 \end{pmatrix}$},
\scalebox{0.8}{ $ \begin{pmatrix} 2 & \frac{-\beta + 1}{2} \\ \frac{\beta + 1}{2} & -3 \end{pmatrix}$}
}\vspace{1mm}\\  
 \hline \rule{0pt}{4mm}
$37$ & {\scalebox{0.8}{ $ \begin{pmatrix} \frac{-\beta + 5}{2} & -\beta - 3 \\ \beta - 6 & \frac{\beta - 7}{2} \end{pmatrix}$},
\scalebox{0.8}{ $ \begin{pmatrix} 2 & \frac{-\beta - 3}{2} \\ \frac{\beta - 3}{2} & -3 \end{pmatrix}$},
\scalebox{0.8}{ $ \begin{pmatrix} \frac{-\beta + 7}{2} & -\beta + 3 \\ \frac{\beta - 5}{2} & \frac{\beta - 9}{2} \end{pmatrix}$},
\scalebox{0.8}{ $ \begin{pmatrix} -3 & \frac{-\beta - 3}{2} \\ \frac{\beta - 3}{2} & 2 \end{pmatrix}$}
}\vspace{1mm}\\  
 \hline \rule{0pt}{4mm}
$40$ & {\scalebox{0.8}{ $ \begin{pmatrix} \beta + 3 & -2\beta - 1 \\ \beta + 3 & -\beta - 4 \end{pmatrix}$},
\scalebox{0.8}{ $ \begin{pmatrix} 2\beta + 8 & -4\beta - 11 \\ 2\beta + 3 & -2\beta - 9 \end{pmatrix}$}
}\vspace{1mm}\\  
 \hline \rule{0pt}{4mm}
$41$ & {\scalebox{0.8}{ $ \begin{pmatrix} \frac{-5\beta + 31}{2} & -5\beta - 16 \\ 5\beta - 32 & \frac{5\beta - 33}{2} \end{pmatrix}$}
}\vspace{1mm}\\  
 \hline \rule{0pt}{4mm}
$44$ & {\scalebox{0.8}{ $ \begin{pmatrix} \beta + 4 & -2\beta - 5 \\ \beta + 2 & -\beta - 5 \end{pmatrix}$},
\scalebox{0.8}{ $ \begin{pmatrix} -\beta - 5 & -2\beta - 5 \\ \beta + 2 & \beta + 4 \end{pmatrix}$}
}\vspace{1mm}\\  
 \hline \rule{0pt}{4mm}
$56$ & {\scalebox{0.8}{ $ \begin{pmatrix} 3 & -\beta - 1 \\ \beta - 1 & -4 \end{pmatrix}$},
\scalebox{0.8}{ $ \begin{pmatrix} -4 & -\beta - 1 \\ \beta - 1 & 3 \end{pmatrix}$}
}\vspace{1mm}\\  
 \hline \rule{0pt}{4mm}
$57$ & {\scalebox{0.8}{ $ \begin{pmatrix} \beta - 8 & -\beta \\ 2\beta - 15 & -\beta + 7 \end{pmatrix}$}
}\vspace{1mm}\\  
 \hline
$69$ & {--}\\ 
 \hline
$105$ & {--}\\ 
 \hline
\end{longtable}
\normalsize

\pagebreak

\paragraph{Hirzebruch--Zagier divisors. }
We list for each $1\leq N\leq 10$ the matrices $B$. The genus ideal is chosen identically to the preceding section.

The number of components may be compared with the formula in \cite[Theorem V.3.3]{VDG}.

\small
\begin{longtable}{ |Sr|Sc|}
\caption{Hirzebruch--Zagier divisors. $\beta^2=$ square-free part of $D$.}\\
\hline 
\multicolumn{1}{|c|}{$D$} & \multicolumn{1}{c|}{$[a,b,\lambda]$'s} \\
\hline\hline
\multicolumn{2}{|c|}{non-principal $g$}\\
\hline
$21$ & \makecell{$F_3$: $[1,-2,3]$, $F_5$: $[1,-15,-8]$, $[1,0,1]$, $F_6$: $[1,-19,9]$}\\ 
\hline
$24$ & \makecell{$F_2$: $[1,-2,-3\beta + 8]$, $F_5$: $[1,-5,5]$, $[1,0,1]$, $F_6$: $[1,24,5\beta + 6]$, $F_8$: $[1,-13,8]$}\\ 
\hline
$28$ & \makecell{$F_3$: $[1,-3,\beta + 6]$, $[1,0,1]$, $F_6$: $[1,36,7\beta + 3]$,\\ $F_7$: $[1,37,-7\beta]$, $[1,-5,-7]$, $F_{10}$: $[1,19,7\beta - 13]$}\\ 
\hline
$33$ & \makecell{$F_2$: $[1,-6,-10]$, $[1,0,1]$, $F_6$: $[1,-2,6]$, $F_8$: $[1,0,2]$, $[1,-10,-13]$}\\ 
\hline
$40$ & \makecell{$F_2$: $[1,14,-5\beta - 8]$, $F_3$: $[1,0,1]$, $[1,-6,9]$, $F_5$: $[1,152,-15\beta + 15]$,\\ $F_8$: $[1,-19,-16]$, $F_{10}$: $[1,127,-13\beta]$}\\ 
\hline
\multicolumn{2}{|c|}{principal $g$}\\
\hline
$29$ & \makecell{$F_1$: $[1,0,1]$, $F_4$: $[1,0,2]$, $F_5$: $[1,-4,11]$, $F_6$: $[1,-2,8]$, $F_7$: $[1,-1,6]$, $F_9$: $[1,0,3]$}\\ 
\hline
$37$ & \makecell{$F_1$: $[1,0,1]$, $F_3$: $[1,-6,15]$, $F_4$: $[1,0,2]$, $F_7$: $[1,-2,9]$, $F_9$: $[1,0,3]$, $F_{10}$: $[1,-3,11]$}\\ 
\hline
$40$ & \makecell{$F_1$: $[1,0,1]$, $[1,-2,9]$, $F_4$: $[1,0,2]$, $F_6$: $[1,50,-15\beta + 16]$,\\ $F_9$: $[1,-1,7]$, $[1,0,3]$, $F_{10}$: $[1,-2,\beta + 10]$}\\ 
\hline
$41$ & \makecell{$F_1$: $[1,0,1]$, $F_2$: $[1,-7,17]$, $F_4$: $[1,0,2]$, $F_5$: $[1,-4,13]$,\\ $F_8$: $[1,-1,7]$, $F_9$: $[1,0,3]$, $F_{10}$: $[1,-6,16]$}\\ 
\hline
$44$ & \makecell{$F_1$: $[1,0,1]$, $[1,-2,\beta + 10]$, $F_4$: $[1,-9,-20]$,\\ $F_5$: $[1,-1,7]$, $[1,0,\beta + 4]$, $F_9$: $[1,-1,\beta + 8]$, $[1,0,3]$}\\ 
\hline
$56$ & \makecell{$F_1$: $[1,0,1]$, $[1,-3,13]$, $F_2$: $[1,2,-7\beta + 24]$,\\ $F_4$: $[1,0,2]$, $F_8$: $[1,-1,8]$, $F_9$: $[1,-2,11]$, $[1,0,3]$}\\ 
\hline
$57$ & \makecell{$F_1$: $[1,0,1]$, $[1,-7,20]$, $F_4$: $[1,0,2]$, $[1,-5,17]$, $F_6$: $[1,-10,24]$,\\ $F_7$: $[1,-1,8]$, $[1,-2,11]$, $F_9$: $[2,0,3]$, $[1,0,3]$}\\ 
\hline
$69$ & \makecell{$F_1$: $[1,0,1]$, $[1,-7,22]$, $F_3$: $[1,-13,30]$, $F_4$: $[1,0,2]$, $[1,-9,25]$,\\ $F_6$: $[1,-2,-12]$, $F_9$: $[2,0,3]$, $[1,0,3]$}\\ 
\hline
$105$ & \makecell{$F_1$: $[1,-8,29]$, $[1,-16,41]$, $[1,0,1]$, $[1,-11,34]$,\\ $F_4$: $[1,0,2]$, $[1,-21,47]$, $[1,-5,23]$, $[1,-13,37]$,\\ $F_9$: $[1,-3,-18]$, $[2,0,3]$, $[1,0,3]$, $[2,-36,87]$}\\ 
\hline
\end{longtable}
\normalsize

\pagebreak
\printbibliography

\end{document}